\documentclass[a4paper,leqno]{amsart}
\UseRawInputEncoding
\usepackage{latexsym}
\usepackage[english]{babel}
\usepackage{fancyhdr}
\usepackage[mathscr]{eucal}
\usepackage{amsmath}
\usepackage{mathrsfs}
\usepackage{amsthm}
\usepackage{amsfonts}
\usepackage{amssymb}
\usepackage{amscd}
\usepackage{bbm}
\usepackage{graphicx}
\usepackage{graphics}
\usepackage{latexsym}
\usepackage{color}
\usepackage{calc}
\usepackage{enumitem}
\usepackage[linktocpage=true,colorlinks=true, linkcolor=blue, citecolor=blue, urlcolor=blue,hyperfootnotes=false]{hyperref}
\usepackage{cleveref}

\usepackage[left=3cm, right=3cm, top=3cm, bottom=3cm]{geometry}
\usepackage{cleveref}
\usepackage{eqnarray}
\usepackage[numbers]{natbib}
\usepackage{mathtools}
\usepackage{slashed}
\usepackage{tikz}
\usetikzlibrary{patterns} 
\usepackage{stmaryrd}
\SetSymbolFont{stmry}{bold}{U}{stmry}{m}{n}

\newcommand{\s}{\boldsymbol{\sigma} \cdot\nabla}

\newcommand{\ri}{_{\rho}^{(j)}}
\newcommand\norm[1]{\left\lVert#1\right\rVert}

\newcommand{\ud}{\mathrm{d}}

\newcommand{\D}{\mathcal{D}}

\newcommand{\Dirac}{\slashed{\mathcal{D}}}

\newcommand{\C}{\mathbb C}

\newcommand{\N}{\mathbb N}
\newcommand{\R}{\mathbb R}

\newcommand\restr[2]{{
  \left.\kern-\nulldelimiterspace 
  #1 
  \littletaller 
  \right|_{#2} 
  }}

\newcommand{\littletaller}{\mathchoice{\vphantom{\big|}}{}{}{}}

\theoremstyle{plain}
\newtheorem{theorem}{Theorem}[section]
\newtheorem{lemma}[theorem]{Lemma}

\newtheorem{proposition}[theorem]{Proposition}

\theoremstyle{definition}
\newtheorem{definition}[theorem]{Definition}

\newtheoremstyle{iremark}
	{6pt}
	{4pt}
	{\upshape}
	{0pt}
	{\itshape}
	{.}
	{5pt plus 1pt minus 1 pt}
	{\thmname{#1} \thmnumber{\itshape#2}\thmnote{(#3)}}

\theoremstyle{iremark}
\newtheorem{remark}[theorem]{Remark}
\newtheorem*{remark*}{Remark}

\numberwithin{equation}{section}

\newcommand{\mig}[1]{{\color{blue}{#1}}}

\begin{document}
\title[Dirac operators on unbounded domains with infinite corners]{Dirac operators with infinite mass boundary conditions\\ on unbounded domains with infinite corners}
\author[M.~Camarasa]{Miguel Camarasa}
\address[M.~Camarasa]{BCAM - Basque Center for Applied Mathematics. Alameda de Mazarredo 14, 48009 Bilbao (Spain).}
\email{mcamarasa@bcamath.org}

\author[F.~Pizzichillo]{Fabio Pizzichillo}
\address[F.~Pizzichillo]{Departamento de Matemática e Informática Aplicadas a la Ingeniería Civil y Naval,
Universidad Politécnica de Madrid\\
E.~T.~S.~I.~de Caminos, Canales y Puertos,
Calle del Profesor Aranguren 3,
Madrid, 28040 (Spain).}
\email{fabio.pizzichillo@upm.es}


\begin{abstract}
We investigate the self-adjointness of the two dimensional Dirac operator with \emph{infinite mass} boundary conditions on an unbounded domain with an infinite number of corners. 
We prove that if the domain has no concave corners, then the operator is self-adjoint. On the other hand, when concave corners are present, the operator is no longer self-adjoint and self-adjoint extensions can be constructed. 
Among these, we characterize the distinguished extension as the unique one whose domain is included in the Sobolev space $H^{s}$, where $s>1/2$ depends on the amplitude of the corners. Lastly, we study the spectrum of this distinguished self-adjoint extension. 

\end{abstract}

\date{\today}

\maketitle

\vspace{-1cm} 


\section{Introduction}

In this paper we study the two dimensional Dirac operator with \emph{infinite mass} boundary conditions on an unbounded polygon with an infinite number of corners. Specifically, our focus lies in investigating both the self-adjointness and spectral properties. 
While the problem on bounded polygons with finite number of corners is well  understood, a comprehensive description in the context of unbounded sets with infinitely many corners is missing. This is the aim of this work. 

The free Dirac operator was introduced by Paul Dirac in \cite{PD} and it is the Hamiltonian that describes the behaviour of $1/2$-spin particles, such as electrons, in the whole space within the framework of relativistic quantum mechanics.
Mathematically, for $m\geq 0$, that denotes the mass of the particle, the free Dirac operator in $\mathbb{R}^2$ is defined by the differential expression
\[
\Dirac_m:=-i\s +m\sigma_3:=-i(\sigma_1 \partial_{x_1}+\sigma_2\partial_{x_2})+m\sigma_3,
\]
where $\boldsymbol{\sigma} =(\sigma_1,\sigma_2)$ and $\sigma_3$ are the Pauli matrices
\[
\sigma_1=\begin{pmatrix}
0 & 1 \\
1 & 0 
\end{pmatrix},\hspace{0.3cm}
\sigma_2=\begin{pmatrix}
0 & -i \\
i & 0 
\end{pmatrix},\hspace{0.3cm}
\sigma_3=\begin{pmatrix}
1 & 0 \\
0 & -1 
\end{pmatrix}.
\]

In quantum mechanics, according to Stone's theorem, observable objects are described by self-adjoint operators. The free Dirac operator on $\mathbb{R}^2$ is self-adjoint on the Sobolev space $H^1(\mathbb{R}^2;\mathbb{C}^2)$ and its spectrum is purely essential and equals $(-\infty,-m]\cup [m,+\infty)$. For more details, see the monograph \cite{Thaller92} for the equivalent three-dimensional case. 

In many materials, such as topological insulators, d-wave superconductors or graphene, the two-dimensional Dirac operator is the effective Hamiltonian describing the properties of low energy electrons in such a structure, see for example \cite{2017arXiv170509187B,1367-2630-11-9-095020,PhysRevLett.112.257401}. 
Due to its significance in describing electron properties, exploring this operator's behavior on a domain $\Omega\subsetneq\mathbb{R}^2$ becomes an interesting subject.
Nevertheless, when the Dirac operator is defined on a domain $\Omega\subsetneq\mathbb{R}^2$, the questions about its self-adjointness and its spectrum become more delicate  and some boundary conditions must be considered. 
For instance, the introduction of \textit{quantum-dot} boundary conditions arises by cutting a graphene sheet along the boundary $\partial \Omega$, see \cite{PhysRevB77085423}. 
A specific example of these condition, the \emph{infinite mass} boundary conditions, was first introduced in \cite{1987MIT}. 
The name originates from the fact that this operator can be represented as the limit of a Dirac operator in the whole space, with the mass term outside $\Omega$ tending to infinity; see \cite{StoVug19} and the references therein. 

Mathematically, denoting by $\textbf{n}$ the outward unit normal vector to $\partial\Omega$, the Dirac operator on $\Omega$ with infinite mass boundary conditions is defined as
\[
\mathcal{D}(D):=\{\psi\in H^1\left(\Omega;\C^2\right): \ -i \sigma_3 \cdot \boldsymbol{\sigma} \cdot \textbf{n} \cdot \psi= \psi \text{ on } \partial \Omega\},\quad
D\psi=\Dirac_m \psi.
\]

When $\Omega$ is bounded, the problem of self-adjointness is well understood and it relies on the regularity of the boundary of the domain. 
If $\partial\Omega$ is $C^2$-regular, \cite{BenFouStoVDB17} establishes that $D$ is self-adjoint. 
Yet, relaxing boundary regularity adds complexities to the problem. 

Indeed, if $\Omega$ is a sector, the self-adjointness depends on the opening of the angle $\omega$ at the corner: when $\omega\in (0,\pi)$ the operator is self-adjoint, and when $\omega\in (\pi,2\pi)$ the operator admits infinitely many self-adjoint extensions, see \cite{LeTeOurBon18} for more details.
Similar analyses are presented in \cite{CL20}, exploring the operator with a vertex flip, and in \cite{CGP22}, considering an additional Coulomb-type interaction centered at the vertex. 
The techniques used in these papers depend strongly on the radial symmetry of the domain. 
In \cite{PizzVDB21}, these results are extended to the case of a \emph{curvilinear polygon}, that is, when $\partial\Omega$ is piecewise $C^2$-regular and has a finite number of corners. 
Recently, the self-adjointness of the Dirac operator on a three-dimensional cone has also been studied in \cite{cassano2023selfadjointness}.
Notice that, besides the mathematical interest of reducing the regularity assumptions, considering hypersurfaces with corners finds additional numerical justifications. Simultaneously, experimental inquiries  explore roughness effects on graphene flake properties, unveiling novel quantum phases, see \cite{Chen_2018}.

A similar outcome occurs with $\delta$-shell boundary conditions, a transmission model between domains in both two and three dimensions. Through confinement, these operators can produce Dirac operators with quantum-dot boundary conditions, specifically the Dirac operator with infinite boundary conditions (see \cite[Remark 4.2]{BeHoOuPan}). Notable references include \cite{DitExnSeb89,AMV14,AMV15,AMV16,B22,B22v2,CLM23,benhellal2023selfadjointness} and the surveys \cite{BHSS22,OBP21}. Another approach is explored in \cite{FL23}, which examines the self-adjointness of the operator on a star graph.

The literature on unbounded sets is relatively limited. References such as \cite{Rabinovich21, Rabinovich22, 2023unbounded, frymark2023spectral} discuss the $\delta$-shell on unbounded  curves, while \cite{BorBrKrOB22} explores the infinite mass boundary condition on a tubular neighborhood of an unbounded planar curve.
Despite the similarity in self-adjointness results to the bounded case,
the unboundedness of the domain affects the spectrum and the \emph{essential gap}.

However, the case of an unbounded domain with an infinite number of corners remains poorly understood. We aim to extend previous results to this general class of domains. Recently, there has been growing interest in problems involving unbounded domains with infinitely many corners. A related investigation is presented in \cite{Richard_2022}, where the authors study a discrete Schrödinger operator on a topological crystal with an infinite number of edges. Although the direct treatment of infinite corners is not addressed, \cite{BEHRNDT20181808} examines Hamiltonians with supported interactions at an infinite number of points on the real line. Additionally, \cite{eigenvaluesnonbounded} explores Schrödinger operators with Dirichlet, Neumann, and Robin boundary conditions on unbounded Lipschitz sets. Furthermore, \cite{euler} investigates the 2D Euler equations in a bounded set with both finite and infinitely many corners. In the case of the Dirac operator, our study reveals that the domain's geometry significantly impacts its self-adjointness and spectral properties.

Before stating our results, we need some definitions.
\begin{definition}\label{def:infinite-polygon}
Given $V:=\{v^{(j)}=(v_1^{(j)},v_2^{(j)})\}_{j\in\mathbb{Z}}\subset\R^2$, we say that $V$ is an \emph{infinite set of non-degenerate corners} if for any $j,k\in \mathbb{Z}$
\begin{enumerate}[label=\emph{(\roman*)}]
\item $v_1^{(j)}\leq v_1^{(j+1)}$,
\item the vectors $(v^{(j)}-v^{(j-1)})$ and $(v^{(j+1)}-v^{(j)})$ are linearly independent in $\R^2$,
\item there exists $\rho>0$ such that $|v^{(j)}-v^{(k)}|>3\rho$.

\end{enumerate}
Given $V$ an infinite set of non-degenerate corners, we define a \emph{non-degenerate infinite polygon} as
\begin{equation}\label{definSv}
\Omega_V:=\left\{(x_1,x_2)\in \mathbb{R}^2: \ \exists j\in\mathbb{Z}: \ x_1\in [v_1^{(j)},v_1^{(j+1)}] \text{ and } x_2>
\frac{v_2^{(j+1)}-v_2^{(j)}}{v_1^{(j+1)}-v_1^{(j)}}\cdot(x_1-v_1^{(j)})+v_2^{(j)}\right\}.
\end{equation}
For any $j\in\mathbb{Z}$, we say that the corner $v^{(j)}$ is \emph{convex} if $\Omega_V\cap B_{\rho}(v^{(j)})$ is a convex set. We say that $v^{(j)}$ is \emph{concave} otherwise.

We define $\omega_j$ as the angle at the corner $v^{(j)}$ as
\begin{equation}\label{eq:def-omega_j}
\omega_j:=
\begin{cases}
\arccos\left(
\displaystyle \frac{v^{(j+1)}-v^{(j)}}{|v^{(j+1)}-v^{(j)}|}
\cdot
\displaystyle \frac{v^{(j-1)}-v^{(j)}}{|v^{(j-1)}-v^{(j)}|}
\right)
& \text{if}\ v^{(j)}\ \text{is convex},
\\
\\
2\pi-\arccos\left(
\displaystyle \frac{v^{(j+1)}-v^{(j)}}{|v^{(j+1)}-v^{(j)}|}
\cdot
\displaystyle \frac{v^{(j-1)}-v^{(j)}}{|v^{(j-1)}-v^{(j)}|}
\right)&
 \text{if}\ v^{(j)}\ \text{is concave},
\end{cases}
\end{equation}
where ``$\cdot$'' and ``$|\cdot|$'' denote the euclidean scalar product and the euclidean norm in $\R^2$ respectively. 
Notice that if $v^{(j)}$ is convex, then $0<\omega_j<\pi$; if $v^{(j)}$ is concave, then $\pi<\omega_j<2\pi$. 
Finally we set 
\[
\operatorname{Conc}(V):=\{j\in\mathbb{Z}:v^{(j)}\ \text{is concave}\}.
\]
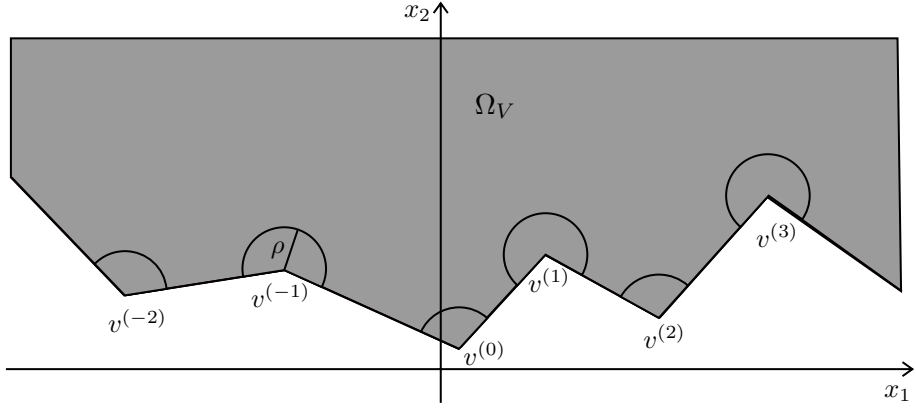
\begin{figure}[ht!!!]\label{infpolygon}
\centering

\tikzset{every picture/.style={line width=0.75pt}} 

\begin{tikzpicture}[x=0.53pt,y=0.53pt,yscale=-1,xscale=1]

\draw  [draw opacity=0][fill={rgb, 255:red, 155; green, 155; blue, 155 }  ,fill opacity=0.38 ] (643.5,36) -- (646,216) -- (551.5,149) -- (475,235) -- (394.5,190) -- (333.5,257) -- (210.5,201) -- (97.5,219) -- (17.5,135) -- (17.5,36) -- cycle ;
\draw [fill={rgb, 255:red, 208; green, 2; blue, 27 }  ,fill opacity=1 ]   (394.5,190) -- (475.5,235) ;
\draw [fill={rgb, 255:red, 208; green, 2; blue, 27 }  ,fill opacity=1 ]   (335,256) -- (395,190) ;
\draw [fill={rgb, 255:red, 208; green, 2; blue, 27 }  ,fill opacity=1 ]   (475.5,235) -- (552.5,148) ;
\draw [fill={rgb, 255:red, 208; green, 2; blue, 27 }  ,fill opacity=1 ]   (552.5,148) -- (646.5,216) ;
\draw [fill={rgb, 255:red, 208; green, 2; blue, 27 }  ,fill opacity=1 ]   (211,201) -- (98,219) ;
\draw [fill={rgb, 255:red, 208; green, 2; blue, 27 }  ,fill opacity=1 ]   (18,135) -- (98,219) ;
\draw [fill={rgb, 255:red, 208; green, 2; blue, 27 }  ,fill opacity=1 ]   (334,257) -- (211,201) ;

\draw  [fill={rgb, 255:red, 155; green, 155; blue, 155 }  ,fill opacity=1 ] (98,187.5) .. controls (113.39,187.5) and (126.03,199.28) .. (127.38,214.32) -- (98,219) -- (76.68,196.61) .. controls (82.05,191) and (89.62,187.5) .. (98,187.5) -- cycle ;
\draw  [fill={rgb, 255:red, 155; green, 155; blue, 155 }  ,fill opacity=1 ] (210.5,170.5) .. controls (226.79,170.5) and (240,183.71) .. (240,200) .. controls (240,204.68) and (238.91,209.11) .. (236.96,213.05) -- (210.5,201) -- (181.53,205.61) .. controls (181.18,203.8) and (181,201.92) .. (181,200) .. controls (181,183.71) and (194.21,170.5) .. (210.5,170.5) -- cycle ;
\draw  [fill={rgb, 255:red, 155; green, 155; blue, 155 }  ,fill opacity=1 ] (334,227.5) .. controls (341.65,227.5) and (348.62,230.41) .. (353.86,235.19) -- (334,257) -- (307.15,244.77) .. controls (311.79,234.58) and (322.07,227.5) .. (334,227.5) -- cycle ;
\draw  [fill={rgb, 255:red, 155; green, 155; blue, 155 }  ,fill opacity=1 ] (395,160.5) .. controls (411.29,160.5) and (424.5,173.71) .. (424.5,190) .. controls (424.5,195.23) and (423.14,200.14) .. (420.75,204.4) -- (395,190) -- (375.14,211.81) .. controls (369.22,206.42) and (365.5,198.64) .. (365.5,190) .. controls (365.5,173.71) and (378.71,160.5) .. (395,160.5) -- cycle ;
\draw  [fill={rgb, 255:red, 155; green, 155; blue, 155 }  ,fill opacity=1 ] (552,118.5) .. controls (568.29,118.5) and (581.5,131.71) .. (581.5,148) .. controls (581.5,154.25) and (579.55,160.05) .. (576.23,164.83) -- (552.5,148) -- (532.67,170.29) .. controls (526.44,164.88) and (522.5,156.9) .. (522.5,148) .. controls (522.5,131.71) and (535.71,118.5) .. (552,118.5) -- cycle ;
\draw  [fill={rgb, 255:red, 155; green, 155; blue, 155 }  ,fill opacity=1 ] (475,204.19) .. controls (482.86,204.19) and (490.03,207.14) .. (495.48,211.98) -- (475,235) -- (448.1,219.97) .. controls (453.38,210.55) and (463.45,204.19) .. (475,204.19) -- cycle ;
\draw    (220,172.5) -- (210.5,201) ;
\draw  (14,271.49) -- (654,271.49)(321,11) -- (321,295.5) (647,266.49) -- (654,271.49) -- (647,276.49) (316,18) -- (321,11) -- (326,18)  ;

\draw (335,249.4) node [anchor=north west][inner sep=0.75pt]    {$v^{( 0)}$};
\draw (380.58,197.71) node [anchor=north west][inner sep=0.75pt]    {$v^{( 1)}$};
\draw (541.55,165.68) node [anchor=north west][inner sep=0.75pt]    {$v^{( 3)}$};
\draw (185,208.39) node [anchor=north west][inner sep=0.75pt]    {$v^{( -1)}$};
\draw (84,226.4) node [anchor=north west][inner sep=0.75pt]    {$v^{( -2)}$};
\draw (462,237.4) node [anchor=north west][inner sep=0.75pt]    {$v^{( 2)}$};
\draw (199,178.4) node [anchor=north west][inner sep=0.75pt]    {$\rho $};
\draw (293,10.4) node [anchor=north west][inner sep=0.75pt]    {$x_{2}$};
\draw (632,281.4) node [anchor=north west][inner sep=0.75pt]    {$x_{1}$};
\draw (343,73.4) node [anchor=north west][inner sep=0.75pt]  [font=\large]  {$\Omega _{V}$};

\end{tikzpicture}
\caption{An example of non-degenerate infinite polygon $\Omega_V$.} 
\end{figure}

\end{definition}
Given $\Omega_V$ a non-degenerate infinite polygon as in \Cref{def:infinite-polygon}, we are interested in the study of the infinite mass Dirac operator acting on $\Omega_V$ defined as
\begin{equation}\label{defdomainDV}
\mathcal{D}(D_V):=\{\psi\in H^1\left(\Omega_V;\C^2\right): \ -i \sigma_3 \, \boldsymbol{\sigma} \cdot \textbf{n} \, \psi= \psi \text{ on } \partial \Omega_V\},\quad
D_V \psi=\Dirac_m \psi.
\end{equation}
Defining a boundary trace from $H^1(\Omega)$ to $H^{1/2}(\partial\Omega)$ in the general case of an unbounded domain $\Omega$ remains unclear. 
However, this property still holds locally. Thus the boundary conditions in \eqref{defdomainDV} have to be meant locally.
Moreover, we prove in \Cref{prop:sigma.nabla=nabla}, that for functions in $\mathcal{D}{(D_V})$ a boundary trace can be defined in $L^2(\partial\Omega_V;\C^2)$.

While a similar analysis can be conducted using Dirac operators with quantum-dot boundary conditions, as described in \cite{PizzVDB21},
we opt for infinite mass boundary conditions because they offer a more direct implementation.

\subsection{Main results}
We are interested in the study of the self-adjointness of the infinite mass Dirac operator on non-degenerate infinite polygons.
The operator $D_V$ is symmetric, as explained in \Cref{prop:sigma.nabla=nabla}.
Before stating the first result of the paper, we need to define the following objects. For any $j\in \operatorname{Conc}(V)$, set
\[
\lambda_j:=\frac{\pi}{2\omega_j}\in \left(\frac{1}{4},\frac{1}{2}\right),
\]
where $\omega_j$ is the corresponding opening angle of the corner $v^{(j)}$, defined in \eqref{eq:def-omega_j}. Let us define a cut-off function $\phi\in C^\infty_c([0,+\infty))$ as
\begin{equation}\label{cutoff}
\phi(r) =  \begin{cases}
  1  & \text{if}\ 0\leq r\leq \frac{1}{2}, \\
 \frac{e^\frac{r-3/4}{(r-1/2)(r-1)}}{1+e^\frac{r-3/4}{(r-1/2)(r-1)}}& \text{if} \  1/2\leq r\leq 1,\\
  0  & \text{if}\ r\geq 1.
  \end{cases} 
\end{equation}
Set moreover
\[
\Phi_j^+(\theta)=\frac{1}{\sqrt{2\omega_j}}\begin{pmatrix}
             e^{i\left(\lambda_{j}-1/2\right)\theta}  
             \\ e^{-i\left(\lambda_{j}-1/2\right)\theta}
             \end{pmatrix}, 
\quad
\Phi_j^-(\theta)=\frac{-i}{\sqrt{2\omega_j}} \begin{pmatrix}
             e^{-i\left(\lambda_{j}+1/2\right)\theta} 
             \\ e^{i\left(\lambda_{j}+1/2\right)\theta}
             \end{pmatrix},
\]
and define, in polar coordinates
\begin{equation}\label{definvarphi}
\varphi_j^+(r,\theta):= 
r^{\lambda_j-1/2}\phi\left(\frac{r}{\rho}\right)\Phi_j^+(\theta), 
\quad
\varphi_j^-(r,\theta):= r^{-\lambda_j-1/2}\phi\left(\frac{r}{\rho}\right)\Phi_j^-(\theta),
\end{equation}
where 
\[
\begin{gathered}
r:=|x-v^{(j)}|\geq 0,\\
\theta=\pi\cdot\left(1-\operatorname{sign}\left(x_2-v^{(j)}_2\right)\right)+\operatorname{sign}\left(x_2-v^{(j)}_2\right)\arccos\left(\dfrac{x_1-v^{(j)}_1}{r}\right)
\in(0,2\pi).
\end{gathered}
\]
We are now ready to state the first result.
\begin{theorem}\label{thm:adjoint}
Let $\Omega_V$ be a non-degenerate infinite polygon, let $D_V$ be  defined  as in \eqref{defdomainDV} and let $D_V^*$ be its adjoint.
Then for all $\psi\in \mathcal{D}(D^*_V)$ there exist $\psi_0\in \mathcal{D}(D_V)$ and $\{c_j^+, \ c_j^-\}_{j\in \operatorname{Conc}(V)}\subset\mathbb{C}^2$ such that 
\begin{equation}\label{eqpsifin}
\psi =\psi_0+\sum_{j\in \operatorname{Conc}(V)}\left(c_j^{+}\varphi_{j}^+ +c_j^{-}  \varphi_{j}^- \right),
\end{equation}
where the series is convergent in the norm of $D^*_V$ if $\operatorname{Conc}(V)$ is infinite.
\end{theorem}

With this result we are able to describe the functions in the domain of $D_V^*$. As a first consequence, \Cref{thm:adjoint} shows that if $\operatorname{Conc}(V)=\emptyset$, then $D_V$ is self-adjoint.
Nevertheless, if $\operatorname{Conc}(V)\neq \emptyset$ there are functions that lie outside $\mathcal{D}(D_V)$ and so, in the general case, one has $D_V\subsetneq D_V^*$. This motivates the study of self-adjoint extensions of $D_V$ depending on the number of concave corners.

For our next result, we need to define the following space: 
\[
\ell^2\left(\frac{1}{\omega_j-\pi}\right):=\left\{\{a_j\}_{j\in \operatorname{Conc}(V)}\subset \mathbb{C}^2: \ \sum_{j\in \operatorname{Conc}(V)}\frac{|a_j|^2}{\omega_j-\pi}<\infty \right\},
\]
where $\omega_j\in(\pi,2\pi)$ is the opening angle of the each concave corner of index $j\in \operatorname{Conc}(V)$. 
Now we are able to characterize  self-adjoint extensions of $D_V$.

\begin{theorem}\label{thmallselfadjointextensions}

Let $\Omega_V$ be a non-degenerate infinite polygon, let $D_V$ be  defined  as in \eqref{defdomainDV} and let $D_V^*$ be its adjoint. Set
\begin{equation}\label{eq:g}
\mathcal{G}:=\begin{cases}
  \mathbb{C}^{|\operatorname{Conc}(V)|} & \text{ if }  |\operatorname{Conc}(V)|<\infty, \\
 	\ell^2(\mathbb{C}^2) & \text{ otherwise},
  \end{cases}
\end{equation}
and
\begin{equation}\label{eq:gtilde}
  \mathcal{\tilde{G}}:=\begin{cases}
  \ell^2\left(\frac{1}{\omega_j-\pi}\right) & \text{if}\ |\operatorname{Conc}(V)|=\infty  \  \text{and}\ w_j\to\pi\ \text{up to a subsequence}, \\
 	\mathcal{G} & \text{otherwise}.
  \end{cases} 
\end{equation}
Define $\Gamma=(\Gamma_+,\Gamma_-):\D(D_V^*)\to \mathcal{G}\oplus\mathcal{G}$ as
\begin{equation}\label{defgamma}
\Gamma_+\psi:=\{i c_j^+\}_{j\in \operatorname{Conc}(V)} \quad \text{and} \quad \Gamma_-\psi:=\{c_j^-\}_{j\in \operatorname{Conc}(V)}, \  \text{ for each } \psi \in \mathcal{D}(D_V^*),
\end{equation}
where $\{c_j^+,c_j^-\}_{j\in \operatorname{Conc}(V)}$ satisfy \eqref{eqpsifin}. Then
\begin{enumerate}[label={(\roman*)}]
\item If $\operatorname{Conc}(V)=\emptyset$, then $D_V$ is self-adjoint. 
\item If $\operatorname{Conc}(V)\neq \emptyset$, then $D_V$ admits an infinite number of self-adjoint extensions $T_U$ defined as 
\begin{equation}\label{definTu}
\mathcal{D}(T_U):=\{\psi\in \mathcal{D}(D_V^*): \ i(\mathbbm{1}-U)\Gamma_+ \psi = (\mathbbm{1}+U)\Gamma_- \psi \}, \quad 
T_U \psi=\Dirac_m \psi, 
\end{equation}
where $U\in \mathcal{U}(\mathcal{G})$ and $U(\tilde{\mathcal{G}})\subset \tilde{\mathcal{G}}$.

\end{enumerate}

\end{theorem}

\begin{remark}
If $\mathcal{G} = \tilde{\mathcal{G}}$, then $(\mathcal{G}, \Gamma_+, \Gamma_-)$ forms a \emph{boundary triple} for $D_V^*$ (see \cite{Behrndt2020BoundaryVP} and \cite{schmüdgen2012unbounded} for definitions). In this case, \eqref{definTu} characterizes all possible self-adjoint extensions of $D_V$; see \cite{BG81} for further details. This scenario reveals some analogies with Coulomb-type interactions as discussed in \cite{CP19}.

However, if $\tilde{\mathcal{G}} \subsetneq \mathcal{G}$, $(\mathcal{G}, \Gamma_+, \Gamma_-)$ is not a boundary triple, so \eqref{definTu} may not describe all self-adjoint extensions. Instead, $(\mathcal{G}, \Gamma_+, \Gamma_-)$ constitutes a \textit{generalized boundary triple}, as outlined in \cite{genboundary}. Specifically, one can define the \(\gamma\)-field \(\gamma\) and the Weyl function $M$ as follows:
\[
\gamma(z) = \left(\Gamma^+ \upharpoonright \operatorname{ker}(D_V^* - z)\right)^{-1}\quad\text{and}\quad
M(z) = \Gamma^- \left(\Gamma^+ \upharpoonright \operatorname{ker}(D_V^* - z)\right)^{-1},
\]
for \(z \in \rho(D_V^*)\). 
Using the approach from \cite{BHSS22}, one can define the extensions $T_\vartheta$ of $D_V$ by
\[
T_\vartheta := D_V^* \upharpoonright \operatorname{ker}(\Gamma_1 - \vartheta \Gamma_0),
\]
where \(\vartheta\) is a self-adjoint operator in $\mathcal{G}$. The self-adjointness of these extensions can be proved using \cite[Section 5]{MR3832009} under certain \emph{decay} conditions on \(\gamma\) and \(M\). 
However, in this paper, we pursue a different strategy because we are unable to derive closed-form expressions for \(\gamma\) and \(M\), which limits our ability to apply the mentioned results.
\end{remark}

In the next result, we characterize the distinguished extension among the infinitely many possible ones.

\begin{theorem}\label{thmH12}
Let $\Omega_V$ be a non-degenerate infinite polygon, let $D_V$ be  defined  as in \eqref{defdomainDV}.
Assume moreover that $\operatorname{Conc}(V)\neq \emptyset$ and that it exists $\epsilon>0$ such that
\begin{equation}\label{2piepsilon}
2\pi-\omega_j>\epsilon \quad \text{for all}\  j\in\text{Conc}(V).
\end{equation}
Then, among all the self-adjoint extensions of $D_V$ defined on \eqref{definTu}, $T_{\mathbbm{1}}$ is the unique one verifying
\[
\D(T_{\mathbbm{1}})\subseteq H^{s}(\Omega_V;\C^2),
\]
for any $s>0$ verifying
\begin{equation}\label{inf s}
s<\underset{j\in \operatorname{Conc}(V)}{\operatorname{inf}}\frac{\pi+\omega_j}{2\omega_j}.
\end{equation} 

\end{theorem}

\begin{remark}
If \eqref{2piepsilon} does not hold, $\omega_j\to 2\pi$ up to a subsequence for $j \to \infty$, and so the localized corner $\Omega_V\cap B_{\rho}(v^{(j)})$ approaches to the whole disk.
In this limit domain, the functions $\varphi_j^+,\varphi_j^-$ defined in \eqref{definvarphi} tend to functions that do not satisfy the infinite mass boundary conditions. Therefore, another approach is needed.  
\end{remark}

\begin{remark}
Since
\[
\inf_{j \in \operatorname{Conc}(V)} \frac{\pi + \omega_j}{2 \omega_j} \in \left[\frac{3}{4}, 1\right],
\]
the extension \(T_{\mathbbm{1}}\) is the unique self-adjoint extension defined by \eqref{definTu} such that
\[
\D(T_{\mathbbm{1}}) \subset H^{1/2}(\Omega_V; \mathbb{C}^2).
\]
The Sobolev space \(H^{1/2}\) is the domain of \(|\Dirac_m|^{1/2}\), representing the space of \emph{finite potential energy} or \emph{energy space} in Dirac theory, as discussed in \cite{CP18}.
\end{remark}

Finally, we study the spectral properties of the Dirac operator defined on an infinite set of non-degenerate corners. For this, we restrict our study to the distingued self-adjoint extension $T_{\mathbbm{1}}$.

\begin{theorem}\label{thm:spectrum}
Let $\Omega_V$ be a non-degenerate infinite polygon. Define
\begin{equation}\label{def: Dsa}
D_{sa}:=\begin{cases}
  D_V & \text{ if }  \operatorname{Conc}(V)=\emptyset, \\
 	T_{\mathbbm{1}} & \text{ otherwise},
  \end{cases}
\end{equation}
where $D_V$ is defined  as in \eqref{defdomainDV} and $T_{\mathbbm{1}}$ as in \eqref{definTu} for $U=\mathbbm{1}$. Assume moreover, up to a rotation or translation, that 
\begin{equation}\label{eq:bdd-condition-spectrum}
\Omega_V\nsubseteq [-R,R]\times [-L,+\infty)\quad\text{for any}\ R,L>0.
\end{equation}
Then:
\begin{equation}\label{eq:spectrum-T1}
\sigma(D_{sa})=\sigma_{ess}(D_{sa})=(-\infty,-m]\cup[m,+\infty).
\end{equation}

\end{theorem}

\begin{remark}\label{rem:bound-spectrum}
The hypothesis \eqref{eq:bdd-condition-spectrum} is essential to get \eqref{eq:spectrum-T1}. 
Indeed, in the case that $\Omega_V:=(-1,1)\times (0,+\infty)$ we get
\[
\sigma(D_V)=\sigma_{ess}(D_V)=
\left(-\infty,-\sqrt{m^2+E_1(m)}\right]\cup\left[\sqrt{m^2+E_1(m)},+\infty\right),
\]
being $E_1(m)$ the unique solution of the equation $m\sin(2\sqrt{E})+\sqrt{E}\cos(2\sqrt{E})=0$ lying in the interval $\left[\frac{\pi^2}{16},\frac{\pi^2}{4}\right)$.
We postpone the proof of this result at the end of \Cref{sec:spectrum}.
\end{remark}

\subsection*{Structure of the paper} 
This paper is organized as follows. In \Cref{sec:adjoint} we characterise the domain of the adjoint operator of $D_V$ and we prove \Cref{thm:adjoint}. 
In Section \Cref{sec:quasi-boundary} we give a detailed description of the self-adjoint extensions of $D_V$ and we prove \Cref{thmallselfadjointextensions}. 
In \Cref{sec:distinguished} we discuss the properties of the \emph{distinguished} self-adjoint extension and we prove \Cref{thmH12}. 
Finally in \Cref{sec:spectrum} we discuss the spectral properties of $D_V$ and we prove \Cref{thm:spectrum}.
For sake of completeness, in \ref{sec:polar} we get a \emph{polar decomposition} of the infinite mass Dirac operator localized close to a corner.

\section{Domain of the adjoint operator and proof of \texorpdfstring{\Cref{thm:adjoint}}{Theorem 1.2}}
\label{sec:adjoint}
This section is devoted to the proof of \Cref{thm:adjoint}.
Without loss of generality we can assume $m=0$. In fact, the multiplication by $m\sigma_3$ is a bounded and symmetric operator in $L^2(\Omega_V;\C^2)$ and it is not relevant to describe the domain of $D_V^*$.

It is well known that when $\Omega$ is a bounded connected Lipschitz domain, the Sobolev space $H^1(\Omega;\C^2)$ is not the maximal domain for the differential expression $\boldsymbol{\sigma} \cdot\nabla$. For this
reason, it is convenient to introduce
\[
\mathcal{K}(\Omega):=\{\psi\in L^2(\Omega;\C^2):\boldsymbol{\sigma} \cdot\nabla \psi\in L^2(\Omega;\C^2)\}
,
\qquad
\norm{\psi}_{\mathcal{K}(\Omega)}=\norm{\psi}_{L^2(\Omega;\C^2)}+\norm{\sigma\cdot\nabla\psi}_{L^2(\Omega;\C^2)}.
\]
Then, by construction, $\D(D_V^*)\subset \mathcal{K}(\Omega_V)$ and $\norm{\psi}_{D_V^*}=\norm{\psi}_{\mathcal{K}(\Omega_V)}$.
We can now prove a useful result.
\begin{proposition}\label{prop:sigma.nabla=nabla}
Let $\Omega_V \subset \mathbb{R}^2$ be a non-degenerate infinite polygon and let $D_V$ be defined as in \eqref{defdomainDV}. Then $D_V$ is a closed and symmetric operator and for any $\psi\in \mathcal{D}(D_V)$ we have 
\begin{gather}
\label{eq:sigma.nabla=nabla}
 \norm{\boldsymbol{\sigma} \cdot\nabla \psi}_{L^2(\Omega_V;\C^2)}=\norm{\nabla \psi}_{L^2(\Omega_V;\C^2)},\\
\label{eq:psiinnerproduct}
 2\operatorname{Re}\langle -i\boldsymbol{\sigma}\nabla\psi,\sigma_3\psi\rangle_{L^2(\Omega_V;\C^2)}=\norm{\psi}^2_{L^2(\partial \Omega_V;\C^2)}. 
\end{gather}
\end{proposition}
\begin{proof}
Let $\psi\in\D(D_V)$, let $\phi$ be the cut-off function defined in \eqref{cutoff} and for $R>0$, define 
\[
\phi_R(x)=\phi\left(\frac{x}{R}\right)
\quad\text{and}\quad
\psi_R:=\phi_R\cdot\psi.
\]
For any $R>0$, $\psi_R\in \D(D_V)$  and, by the dominated convergence theorem, we get that when $R\to +\infty$, $\psi_R\to \psi$ and 
$\boldsymbol{\sigma} \cdot\nabla\psi_R\to \boldsymbol{\sigma} \cdot\nabla\psi$ in $L^2(\Omega_V;\C^2)$.
Moreover, since $\psi_R$ is compactly supported and the piecewise continuous curvature of $\Omega_V$ is equal to $0$, thanks to \cite[Lemma 2.1]{PizzVDB21}  $\psi_R$ verifies \eqref{eq:sigma.nabla=nabla}.
Thus, passing to the limit $R\to+\infty$, $\psi$ verifies \eqref{eq:sigma.nabla=nabla}. This means that the $H^1$-norm and the $D_V$-norm are equivalent on $\D(D_V)$. Then, by the continuity of the trace
\[
\overline{\D(D_V)}^{D_V}= \overline{\D(D_V)}^{H^1}\subset 
\D(D_V),
\]
thus $D_V$ is closed.
Moreover, by \cite[Lemma 2.1]{PizzVDB21} we have
\[
\begin{split}
 2\operatorname{Re}\langle -i\boldsymbol{\sigma}\cdot\nabla\psi_R,\sigma_3\psi_R \rangle_{L^2(\Omega_V;\C^2)}
&=
\int_{\partial\Omega_V} \langle -i\boldsymbol{\sigma} \cdot \textbf{n}\psi_R,\sigma_3\psi_R\rangle_{\C^2}
=
\norm{\psi_R}^2_{L^2(\partial \Omega_V;\C^2)},\\
\langle -i\boldsymbol{\sigma}\cdot\nabla\psi_R,\psi_R \rangle_{L^2(\Omega_V;\C^2)}
&=
\langle \psi_R, -i\boldsymbol{\sigma}\cdot\nabla\psi_R \rangle_{L^2(\Omega_V;\C^2)}.
\end{split}
\]
where in the first equation we have used the fact that $\sigma_3$ is a symmetric matrix and $\psi_R$ verifies the infinite mass boundary conditions $-i\sigma_3\boldsymbol{\sigma}\cdot\textbf{n}\psi_R=\psi_R$ on $\partial\Omega_V$. 
Thus, passing to the limit $R\to+\infty$, we deduce that $\psi$ verifies \eqref{eq:psiinnerproduct} and $D_V$ is a symmetric operator.

\end{proof}

We want to describe the space $\D(D_V^*)$. It is known that the self-adjointness of the Dirac operator defined on a single wedge depends on its convexity. When it is concave, it admits an infinite number of self-adjoint extensions (see \cite{LeTeOurBon18}, \cite{PizzVDB21}). Therefore, since the problem is local, it is useful to localize the functions on $\D(D_V^*)$ near to the concave corners.

\begin{proposition}\label{prop:convergnce-in-op-norm}
Let $\phi$ be the cut-off function defined in \eqref{cutoff} and for $j\in \operatorname{Conc}(V)$, define the 
cut-off functions localized on each concave corner $v^{(j)}$ as 
\[
\phi_{\rho}^{(j)}(x)=\displaystyle \phi\left(\frac{|x-v^{(j)}|}{\rho}\right).
\]
Let $\psi\in\mathcal{D}(D_V^*)$ and define
\begin{equation}\label{psirho}
\psi_{\rho}^{(j)}:=\psi \cdot \phi_{\rho}^{(j)} \quad  \text{and} \quad \chi_1=\psi-\sum_{j\in \operatorname{Conc}(V)}\psi_{\rho}^{(j)}.
\end{equation}
Then $\chi_1\in\D(D_V)$ and
\[
\psi=\chi_1+\sum_{j\in \operatorname{Conc}(V)}\psi_{\rho}^{(j)},
\]
where the series is convergent in $\D(D_V^*)$ if $\operatorname{Conc}(V)$ is infinite.
\end{proposition}

\begin{proof}
Let us firstly show that $\chi_1\in\D(D_V)$. Let $\phi$ be the cut-off function defined in \eqref{cutoff} and for $R>0$, define 
\[
\phi_R(x)=\phi\left(\frac{x}{R}\right)
\quad\text{and}\quad
\chi_{1,R}:=\phi_R\cdot\chi_1.
\]
For any $R>0$, $\chi_{1,R}$ is compactly supported and it is localised far away from any concave corner, then  $\chi_{1,R}\in \mathcal{D}(D_V)$, by \cite[Theorem 1.6]{PizzVDB21}. 
Moreover, by the dominated convergence theorem we get that $\chi_{1,R}\to \chi_1$ and 
$\boldsymbol{\sigma} \cdot\nabla\chi_{1,R}\to \boldsymbol{\sigma} \cdot\nabla\chi_1$ in $L^2(\Omega_V;\C^2)$. 
Then we conclude that $\chi_1\in \D(D_V)$ thanks to \Cref{prop:sigma.nabla=nabla}.

If $\operatorname{Conc}(V)$ is finite, the proof is complete. 
Let us assume now that $\operatorname{Conc}(V)$ is infinite. Then, by construction the series is point-wise convergent. Using the fact that the supports of the functions $\psi_\rho^{(j)}$ are disjoint, we have
\[
\begin{split}
\sum_{j\in \operatorname{Conc}(V)}
\norm{\psi_{\rho}^{(j)}}_{L^2(\Omega_V;\C^2)}^2\
&=
\sum_{j\in \operatorname{Conc}(V)}\norm{\phi_{\rho}^{(j)}\cdot \psi}_{L^2(\Omega_V;\C^2)}^2
\leq 
\sum_{j\in \operatorname{Conc}(V)}\norm{\psi}_{L^2\left(\Omega_V\cap B_{\rho}^{(j)};\C^2\right)}^2
\\
&\leq \norm{\psi}_{L^2(\Omega_V;\C^2)}^2.
\end{split}
\]
Being $L^2(\Omega_V;\C^2)$ a Banach space, this implies that the series is convergent in $L^2(\Omega_V;\C^2)$. Moreover
\[
\begin{split}
\sum_{j\in \operatorname{Conc}(V)}\norm{D_V^* \psi_{\rho}^{(j)}}_{L^2(\Omega_V;\C^2)}^2 
&\leq
C_{\rho}\sum_{j\in \operatorname{Conc}(V)}\left(
\norm{\s\psi}_{L^2(\Omega_V\cap B_{\rho}^{(j)};\C^2)}^2
+\norm{\psi}_{L^2(\Omega_V\cap B_{\rho}^{(j)};\C^2)}^2
\right)
\\
&\leq  
C_{\rho} 
\left(\norm{D_V^*\psi}_{L^2(\Omega_V;\C^2)}^2
+\norm{\psi}_{L^2(\Omega_V;\C^2)}^2\right),
\end{split}
\]
where $C_\rho>0$ only depends on $\rho$.
Thus we can conclude that the series is convergent in $\mathcal{D}(D^*_V)$.
\end{proof}

We now describe $\psi_\rho^{(j)}$ defined in \eqref{psirho}. At this purpose, for sake of completeness, we recall some known results. 
\begin{lemma}\label{definlocal}
Let $\Omega_V$ be defined as in \eqref{definSv}.
For $j\in \operatorname{Conc}(V)$, define the localized Dirac operator on $\Omega_V\cap B_{\rho}(v^{(j)})$ as
\begin{equation}\label{eq:def-D-rho-j}
\mathcal{D}(D_{\rho}^{(j)}):=\{u\in H^1(\Omega_V\cap B_{\rho}(v^{(j)}); \mathbb{C}^2): \ u^E\in \mathcal{D}(D_V) \}\quad\text{and}\quad
     D_{\rho}^{(j)} u:= -i\s u, 
\end{equation}
where $u^{E}$ means the extension by zero of $u$. 
Moreover define
\begin{equation}\label{Nrho}
N_{\rho}^{(j)}:=
\begin{Bmatrix}
w\in \mathcal{D}\left(D_{\rho}^{(j)*}\right):\ \Delta w=0, \\
 -i \sigma_3 \cdot \boldsymbol{\sigma} \cdot \vec{n} \cdot w= w \text{ and } -i \sigma_3 \cdot \boldsymbol{\sigma} \cdot \vec{n} \cdot D_{\rho}^{(j)}w= D_{\rho}^{(j)}w \text{ on } \partial \Omega_V\cap B_{\rho}(v^{(j)})
 \end{Bmatrix}.
\end{equation}
Then
\begin{enumerate}[label={(\roman*)}]
\item $D_{\rho}^{(j)}$ is closed, symmetric and it has closed range.
\item $\operatorname{Ker}\left(\left({D_{\rho}^{(j)}}^*\right)^2\right)=N_{\rho}^{(j)}$.
\item For any $v\in\mathcal{D}(D_{\rho}^{(j)})$
\begin{equation}\label{eq:estimate-bessel-zero}
\norm{D_\rho^{(j)} v}_{L^2(\Omega_V\cap B_{\rho}(v^{(j)}); \mathbb{C}^2)}\geq \frac{2}{\rho}\norm{v}_{L^2(\Omega_V\cap B_{\rho}(v^{(j)}); \mathbb{C}^2)}.
 \end{equation} 
\end{enumerate}
\end{lemma}
The proof of this result can be find in \cite[Proof of Lemma 2.7]{PizzVDB21}, except for \eqref{eq:estimate-bessel-zero} that is proved in \Cref{sec:polar}.

Before stating the next result, we need to introduce the following set:
\[
\mathcal{N}_{\rho}^{(j)}:=\left\{\phi\left(\frac{|x-v^{(j)}|}{\rho}\right) w^E(x): \ w \in N_{\rho}^{(j)}\right\},
\]
where $N_{\rho}^{(j)}$ is defined in \eqref{Nrho} and $w^E$ means the extension by zero to $\Omega_V$ of $w$.

\begin{proposition}\label{prop:sum-psi-j}
Let $\psi\in \mathcal{D}(D_V^*)$ and for $j\in \operatorname{Conc}(V)$, let $\psi_{\rho}^{(j)}$ be defined as in \eqref{psirho}.
Then there exist $\chi_2\in \D(D_V)$ and $w_{\rho}^{(j)}\in \mathcal{N}_{\rho}^{(j)}$ such that
\[
\sum_{j\in \operatorname{Conc}(V)}\psi_{\rho}^{(j)}=\chi_2+\sum_{j\in \operatorname{Conc}(V)}w_{\rho}^{(j)},
\]
where the series is convergent in $\D(D_V^*)$ if $\operatorname{Conc}(V)$ is infinite.
\end{proposition}

\begin{proof}
With abuse of notation, we denote by $\psi_\rho^{(j)}={{\psi_\rho^{(j)}}_|}_{\Omega_V\cap B_\rho(v^{(j)})}$. 

By definition $\psi_\rho^{(j)}\in \D\left({D_\rho^{(j)}}^*\right)$.
By \Cref{definlocal}, ${D_{\rho}^{(j)}}^{-1}:\operatorname{Ran}(D_{\rho}^{(j)})\to \D(D_{\rho}^{(j)})$ is well-defined and bounded and $\operatorname{Ran}(D_{\rho}^{(j)})$ is a closed subspace of $L^2(\Omega_V\cap B_\rho(v^{(j)});\C^2)$. We decompose $\s \psi_\rho^{(j)}$ by projecting on this subspace and its orthonormal, that is
\begin{equation}\label{eq:orth-dec}
\s \psi_\rho^{(j)}=\s u_\rho^{(j)}+\tilde w_\rho^{(j)},
\end{equation}
with $u_\rho^{(j)}\in\D(D_\rho^{(j)})$,  and $\tilde w_\rho^{(j)}\in\operatorname{Ran}(D_\rho^{(j)})^\perp=\operatorname{Ker}\left({D_\rho^{(j)}}^*\right)$.
Define $w_\rho^{(j)}:=\psi_\rho^{(j)}-u_\rho^{(j)}\in \D\left({D_\rho^{(j)}}^*\right)$ and then, thanks to \Cref{definlocal} we have that $w_\rho^{(j)}\in N_\rho^{(j)}$. Denoting, with abuse of notation, by $u_\rho^{(j)}$ and $w_\rho^{(j)}$ their extension by zero, we can conclude that 
\[
\psi_\rho^{(j)}=u_\rho^{(j)}+ w_\rho^{(j)},
\]
with $u_\rho^{(j)}\in\D(D_V)$ and $w_\rho^{(j)}\in \mathcal{N}_\rho^{(j)}$.
Thus
\begin{equation}\label{eq:sum-to-split}
\sum_{j\in \operatorname{Conc}(V)}\psi_{\rho}^{(j)}=\sum_{j\in \operatorname{Conc}(V)}\left(u_\rho^{(j)}+w_{\rho}^{(j)}\right).
\end{equation}
If $\operatorname{Conc}(V)$ is finite, the sum can be split. Since $u_\rho^{(j)}\in \D(D_V)$ for any $j\in \operatorname{Conc}(V)$, we easily have that
$\chi_2:=\sum_{j\in \operatorname{Conc}(V)}u_\rho^{(j)}\in \D(D_V)$.
Let us assume that $\operatorname{Conc}(V)$ is infinite. 
In this case, the series in \eqref{eq:sum-to-split} is convergent in $\D(D_V^*)$. 
Let us show that it can be split as the sum of two convergent series.
Indeed, since the decomposition on \eqref{eq:orth-dec} is orthogonal with $\tilde w_{\rho}^{(j)}=\s w_{\rho}^{(j)}$, we have that
$\sum_{j\in \operatorname{Conc}(V)}\s u_\rho^{(j)}$ and $\sum_{j\in \operatorname{Conc}(V)}\s w_{\rho}^{(j)}$ are convergent in $
L^2(\Omega_V;\C^2)$. Applying \eqref{eq:estimate-bessel-zero} we have
\[
\sum_{j\in \operatorname{Conc}(V)}\norm{\s u_\rho^{(j)}}_{L^2(\Omega_V;\C^2)}\geq \frac{2}{\rho}\sum_{j\in \operatorname{Conc}(V)} \norm{u_\rho^{(j)}}_{L^2(\Omega_V;\C^2)}.
\]
Thank to this, we can conclude that that $\sum_{j\in \operatorname{Conc}(V)}u_\rho^{(j)}$ is convergent in $\mathcal{D}(D_V^*)$ and so does $\sum_{j\in \operatorname{Conc}(V)}w\ri$.
To conclude the proof, set $\chi_2:=  \sum_{j\in \operatorname{Conc}(V)}  u\ri\in \D(D_V)$ thanks to \Cref{prop:sigma.nabla=nabla}.
\end{proof}

We have just shown that a function in $\D(D_V^*)$ can be written as the sum of functions in $\mathcal{D}(D_V)$ plus of some harmonic functions localized close to any concave corner that verify some boundary conditions. That is
\begin{equation}\label{eq:almost-done-dec}
\psi=\chi_1+\chi_2+\sum_{j\in \operatorname{Conc}(V)}w_{\rho}^{(j)},
\quad
\chi_1,\chi_2\in\D(D_V),\, w_{\rho}^{(j)}\in\mathcal{N}_{\rho}^{(j)},
\end{equation}
where the series is convergent in $\D(D_V^*)$ if $\operatorname{Conc}(V)$ is infinite.
The final step to prove \Cref{thm:adjoint} is to characterize the functions $w_{\rho}^{(j)}$.

\begin{proof}[\Cref{thm:adjoint}:] 
For any $j\in\text{Conc}(V)$, thanks to \cite[Theorem 3.2]{PizzVDB21} we have that $w_{\rho}^{(j)}$ can be decomposed as the orthogonal sum in $\D(D_V^*)$
\begin{equation}\label{decomwrhoi} 
w\ri=
\chi\ri +
c_j^+\,\varphi_j^+ +c_j^-\,\varphi_j^-,
\end{equation}
where $\chi\ri\in \D(D_V)$, $\varphi_j^+,\varphi_j^-$ are defined in \eqref{definvarphi} and $c_j^+,c_j^-\in\C$. Notice that the statement of \cite[Theorem 3.2]{PizzVDB21} does not specify that this decomposition is orthogonal. Nevertheless,  this is clear from the proof, since the functions in \eqref{decomwrhoi} belong to different eigenspaces.
Then
\begin{equation}
\label{eq:sum-to-split2}
\sum_{j\in \operatorname{Conc}(V)}
w\ri=
\sum_{j\in \operatorname{Conc}(V)}
\left(
\chi\ri +
c_j^+\,\varphi_j^+ +c_j^-\,\varphi_j^-
\right).
\end{equation}
If $\operatorname{Conc}(V)$ is finite we can split the sum. Since $\tilde w\ri\in \D(D_V)$ for any $j\in \operatorname{Conc}(V)$, we easily have that
$\chi_3:=\sum_{j\in \operatorname{Conc}(V)}\tilde{w}_{\rho}^{(j)}\in \D(D_V)$. Thus, thanks to \eqref{eq:almost-done-dec} setting $\psi_0=\chi_1+\chi_2+\chi_3$, the proof is concluded. 
Let us now assume that $\operatorname{Conc}(V)$ is infinite.
In this case, the series in \eqref{eq:sum-to-split2} is convergent in $\D(D_V^*)$. Since the decomposition on \eqref{decomwrhoi} is orthogonal in $\D(D_V^*)$, we have that the series \eqref{eq:sum-to-split2} can be split as
\[
\sum_{j\in \operatorname{Conc}(V)}
w\ri=
\sum_{j\in \operatorname{Conc}(V)}
\chi\ri +
\sum_{j\in \operatorname{Conc}(V)}
\left(
c_j^+\,\varphi_j^+ +c_j^-\,\varphi_j^-
\right),
\]
where both series on the right-hand-side are convergent in $\D(D_V^*)$. 
Then $\chi_3:=\sum_{j\in \operatorname{Conc}(V)}
\chi\ri \in\D(D_V)$ by \Cref{prop:sigma.nabla=nabla}. Thanks to this and \eqref{eq:almost-done-dec}, setting $\psi_0=\chi_1+\chi_2+\chi_3$ the proof is concluded.  
\end{proof}

\section{Description of self-adjoint extensions and proof of \texorpdfstring{\Cref{thmallselfadjointextensions}}{Theorem 1.3}}
\label{sec:quasi-boundary}

In this section we present some self-adjoint extensions of the Dirac operator \eqref{defdomainDV} in an non-degenerate infinite polygon $\Omega_V$. Denote by $\mathcal{G}$, $\tilde{\mathcal{G}}$ the spaces defined in \eqref{eq:g} and \eqref{eq:gtilde}, respectively. We present self-adjoint extensions that are in correspondence with the unitary operators $U\in\mathcal{U}(\mathcal{G})$ that leave $\tilde{\mathcal{G}}$ invariant. We also assume $|\operatorname{Conc}(V)|=+\infty$ since the description of self-adjoint extensions with a finite number of concave corners is well understood (see \cite{PizzVDB21}). At this purpose, we need the following result: 
\begin{lemma}\label{lemma:boundarytriples} Let $\Gamma=(\Gamma_+$, $\Gamma_-)$ be defined as in \eqref{defgamma}. Then,

\begin{enumerate}[label=(\roman*)]
\item\label{item1lemma:boundarytriples} $\textup{Ker}(\Gamma)$ is dense in $L^2(\Omega_V;\C^2)$,
\item\label{item2lemma:boundarytriples} $\textup{Ran}(\Gamma_+)=\mathcal{G}$ and $\textup{Ran}(\Gamma_-)=\tilde{\mathcal{G}}$,
\item \label{item3lemma:boundarytriples}For all $\psi,\tilde{\psi} \in \mathcal{D}(D_V^*)$, the Green identity holds
\[
\langle D_V^* \psi,\tilde{\psi} \rangle_{L^2(\Omega_V;\C^2)}-\langle \psi , D_V^* \tilde{\psi} \rangle_{L^2(\Omega_V;\C^2)}= \langle \Gamma_+\psi,\Gamma_-\tilde{\psi} \rangle_{\mathcal{G}}-\langle \Gamma_-\psi , \Gamma_+ \tilde{\psi} \rangle_{\mathcal{G}}.
\]
\end{enumerate}
\end{lemma}

\begin{proof}

In this proof $C$ refers to any positive constant.

First, we prove \emph{\ref{item1lemma:boundarytriples}}.
Let $\psi\in \mathcal{D}(D_V^*)$. 
By \Cref{thm:adjoint}, we have 
\begin{equation}\label{proofpsidense}
\psi =\psi_0+\displaystyle \sum_{j\in \operatorname{Conc}(V)}\left(c_j^{+}\varphi_{j}^+ +c_j^{-}  \varphi_{j}^- \right).
\end{equation}
If $\Gamma \psi =0$, then $c_j^+=c_j^-=0$ for all $j\in \operatorname{Conc}(V)$ and so $\ker(\Gamma)=\D(D_V)$. Since $\mathcal{C}_c^{\infty}(\Omega_V;\mathbb{C}^2)\subset \D(D_V) \subset L^2(\Omega_V,\C^2)$ and $\mathcal{C}_c^{\infty}(\Omega_V;\mathbb{C}^2)$ is a dense subspace of $L^2(\Omega_V;\mathbb{C}^2)$, so does $\D(D_V)$. 
 
Let us prove \emph{\ref{item2lemma:boundarytriples}}.
Let $\psi\in \mathcal{D}(D_V^*)$ and let us decompose it as in \eqref{proofpsidense}.
Since $\{\varphi_j^+,\varphi_j^-\}_{j\in \operatorname{Conc}(V)}$ are orthogonal functions in $L^2(\Omega_V;\C^2)$, we have
\[
\norm{\psi-\psi_0}_{L^2(\Omega_V;\mathbb{C}^2)}^2=
\sum_{j\in \operatorname{Conc}(V)}
\left(
|c_j^{+}|^2
\norm{
\varphi_{j}^+ 
}_{L^2(\Omega_V;\mathbb{C}^2)}^2
+
|c_j^{-}|^2\norm{
\varphi_{j}^- 
}_{L^2(\Omega_V;\mathbb{C}^2)}^2
\right).
\]
Using polar coordinates
\[
\begin{split}
\norm{
\varphi_{j}^\pm
}_{L^2(\Omega_V;\mathbb{C}^2)}^2
&=
\int_0^\rho
r^{\pm 2 \lambda_j}\phi\left(\frac{r}{\rho}\right)^2\,dr
\geq
\int_0^{\rho/2}
r^{\pm 2 \lambda_j}\,dr
=
\frac{(\rho/2)^{1\pm 2\lambda_j}}{1\pm 2\lambda_j}.
\end{split}
\]
Being $2\lambda_j=\pi/\omega_j$ for $\omega_j\in(\pi,2\pi)$, we can conclude that there exists $C_\rho>0$ only depending on $\rho$ such that
\[
\norm{\psi-\psi_0}_{L^2(\Omega_V;\mathbb{C}^2)}^2\geq 
C_\rho
\sum_{j\in \operatorname{Conc}(V)}
\left(
|c_j^+|^2+\frac{|c_j^-|^2}{\omega_j-\pi}
\right).
\]
This proves that $\Gamma_+\psi \in \mathcal{G}$ and $\Gamma_-\psi \in \tilde{\mathcal{G}}$. Let us prove the other inclusion. Let $\{c_j^+\}_{j\in \operatorname{Conc}(V)}\in\mathcal{G}$ and $\{c_j^-\}_{j\in \operatorname{Conc}(V)}\in\tilde{\mathcal{G}}$ and define
\[
\psi := \sum_{j\in \operatorname{Conc}(V)}\left(c_j^{+}\varphi_{j}^+ +c_j^{-}  \varphi_{j}^- \right).
\]
Reasoning as above, we have that 
\[
\norm{
\varphi_{j}^\pm
}_{L^2(\Omega_V;\mathbb{C}^2)}^2
\leq 
\frac{\rho^{1\pm 2\lambda_j}}{1\pm 2\lambda_j}.
\]
Then
\[
\norm{\psi}_{L^2(\Omega_V;\mathbb{C}^2)}^2 
\leq 
C_\rho
\sum_{j\in \operatorname{Conc}(V)} 
\left(
|c_j^+|^2
+
\frac{|c_j^-|^2}{\omega_j-\pi}
\right),
\]
with $C_\rho>0$ only depending on $\rho$. Then $\psi\in L^2(\Omega_V;\C^2)$. Let us prove that $\s \psi \in L^2(\Omega_V;\C^2)$. Since the support of the functions $\{\varphi_j^+,\varphi_j^-\}_{j\in\text{Conc}(V)}$ are disjoint, we have 
\[
\norm{\s \psi}_{L^2(\Omega_V;\mathbb{C}^2)}^2
= \sum_{j\in \operatorname{Conc}(V)}\norm{D_V^*\left(c_j^+\varphi_j^++c_j^-\varphi_j^-\right)}_{L^2(\Omega_V;\mathbb{C}^2)}^2.
\]
By \eqref{eq:compute.HV}, in polar coordinates we get
\begin{equation}\label{proof:diracpolar}
D_V^*\left(c_j^+\varphi_j^+(r,\theta)+c_j^-\varphi_j^-(r,\theta)\right)
=
-\frac{1}{\rho}\phi'\left(\frac{r}{\rho}\right)
\left(
c_j^+r^{\lambda_j-1/2}\,\Phi_j^-(\theta)+
c_j^-r^{-\lambda_j-1/2}\,\Phi_j^+(\theta)
\right).
\end{equation}
Therefore, since $\Phi_j^+$ and $\Phi_j^-$ are orthogonal
\[
\begin{split}
\norm{D_V^*\left(c_j^+\varphi_j^++c_j^-\varphi_j^-\right)}_{L^2(\Omega_V;\mathbb{C}^2)}^2
&=
\frac{1}{\rho}
\int_{\rho/2}^\rho
|c_j^+|^2
r^{2\lambda_j}\phi'\left(\frac{r}{\rho}\right)^2+
|c_j^-|^2
r^{-2\lambda_j}\phi'\left(\frac{r}{\rho}\right)^2\,dr
\\
&
\leq 
C_\rho\left(
\frac{|c_j^+|^2}{1+2\lambda_j}+\frac{|c_j^-|^2}{1-2\lambda_j}
\right),
\end{split}
\]
being $C_\rho>0$ only depending on $\rho$. Thus, replacing the value of $\lambda_j$ we get
\[
\norm{\s \psi}_{L^2(\Omega_V;\mathbb{C}^2)}^2
\leq 
C_\rho
\sum_{j\in \operatorname{Conc}(V)}
\left(
|c_j^+|^2+\frac{|c_j^-|^2}{\omega_j-\pi}.
\right).
\]
Given that $\{c_j^+\}_{j\in \operatorname{Conc}(V)}\in\mathcal{G}$ and $\{c_j^-\}_{j\in \operatorname{Conc}(V)}\in\tilde{\mathcal{G}}$, we have that $\s \psi \in L^2(\Omega_V;\C^2)$. 
Finally, since $D_V^*$ is a closed operator and $\varphi_j^+, \varphi_j^-\in\D(D_V^*)$, we get $\psi\in \D(D_V^*)$. 

To conclude the proof, we prove \emph{\ref{item3lemma:boundarytriples}}. 
Let $\psi,\tilde\psi \in \mathcal{D}(D_V^*)$. By \Cref{thm:adjoint},
\[
\psi =\psi_0+\displaystyle \sum_{j\in \operatorname{Conc}(V)}\left(c_j^{+}\varphi_{j}^+ +c_j^{-}  \varphi_{j}^- \right), \ \tilde \psi =\tilde \psi_0+\displaystyle \sum_{j\in \operatorname{Conc}(V)}\left(\tilde{c_j}^+\varphi_{j}^+ +\tilde{c_j}^-  \varphi_{j}^- \right).
\]
Since $\psi_0,\tilde{\psi}_0\in \D(D_V)$, $\varphi_{j}^+,\varphi_{j}^-\in \D(D_V^*)$, the support of the functions $\{\varphi_j^+, \varphi_j^-\}_{j\in\text{Conc}(V)}$ are disjoint and due to the symmetry of $D_V$, we obtain
\[
\begin{split}
\langle 
\D_V^*\psi,\tilde \psi 
\rangle_{L^2(\Omega_V;\C^2)}
-&
\langle
\psi, D_V^* \tilde\psi
\rangle_{L^2(\Omega_V;\C^2)}
\\
=& \sum_{j\in \operatorname{Conc}(V)} \langle  D_V^*\left(c_j^{+}\varphi_{j}^+ +c_j^{-}  \varphi_{j}^-\right),\tilde{c_j}^+\varphi_{j}^+ +\tilde{c_j}^-  \varphi_{j}^- \rangle_{L^2(\Omega_V;\C^2)}\\
&+ \sum_{j\in \operatorname{Conc}(V)} \langle c_j^{+}\varphi_{j}^+ +c_j^{-}  \varphi_{j}^-, D_V^*\left(\tilde{c_j}^+\varphi_{j}^+ +\tilde{c_j}^-  \varphi_{j}^- \right)\rangle_{L^2(\Omega_V;\C^2)}.
\end{split}
\]
Thanks to \eqref{proof:diracpolar} and since $\Phi_j^+$ and $\Phi_j^-$ are orthonormal functions, it follows that
\[
\begin{split}
\langle D_V^*\varphi_j^+, \varphi_j^+ \rangle_{L^2(\Omega_V;\C^2)}&=\langle D_V^*\varphi_j^-, \varphi_j^- \rangle_{L^2(\Omega_V;\C^2)}=0,\\
\langle
D_V^*\varphi_j^+, \varphi_j^- \rangle_{L^2(\Omega_V;\C^2)}&=\langle
D_V^*\varphi_j^-, \varphi_j^+ \rangle_{L^2(\Omega_V;\C^2)}=\frac{i}{2}.
\end{split}
\]
Finally, we get 
\[\begin{split}
\langle \D_V^*\psi,\tilde \psi \rangle_{L^2(\Omega_V;\C^2)} -\langle \psi, D_V^* \tilde\psi\rangle_{L^2(\Omega_V;\C^2)} & = \sum_{j\in \operatorname{Conc}(V)} \left(ic_j^+\overline{\tilde{c_j}^-}+i c_j^-\overline{\tilde{c_j}^+}\right) \\
&=  \langle \Gamma_+\psi,\Gamma_- \tilde \psi \rangle_{\mathcal{G}} -\langle \Gamma_-\psi, \Gamma_+ \tilde\psi\rangle_{\mathcal{G}}.
\end{split}
\]
\end{proof}

Before proving \Cref{thmallselfadjointextensions}, we need the following result.

\begin{lemma}\label{lemma: gammaU}
Let $\mathcal{D}(T_U)$ be defined as in \eqref{definTu}. Then, for any $\psi\in \mathcal{D}(T_U)$, there exists a unique $\mathfrak{g}\in \mathcal{G}$ such that
\begin{equation}\label{eq:Gamma}
\Gamma\psi:=(\Gamma_+\psi,\Gamma_-\psi)=\left((\mathbbm{1}+U)\mathfrak{g};  \ i(\mathbbm{1}-U)\mathfrak{g}\right).
\end{equation}
Moreover, the map $\mathfrak{G}:\psi\in \mathcal{D}(T_U)\mapsto \mathfrak{g}\in \mathcal{G}$ such that \eqref{eq:Gamma} is verified has dense range. 
\end{lemma}
\begin{proof}
Define 
\[
\Lambda:=\{\left((\mathbbm{1}+U)\mathfrak{g};  \ i(\mathbbm{1}-U)\mathfrak{g}\right): \ \mathfrak{g}\in\mathcal{G}\}\subset \mathcal{G}\oplus \mathcal{G},
\]
and 
\[
\Pi:=\{(\mathfrak{g}_1,\mathfrak{g}_2)\in \mathcal{G}\oplus \mathcal{G}: \  i(1-U)\mathfrak{g}_1= (1+U)\mathfrak{g}_2\}.
\]
Let $\psi\in \mathcal{D}(T_U)$. By construction, $\Gamma \psi \in \Pi$.
To prove that $\Gamma\psi\in\Lambda$,  we show that $\Pi=\Lambda$.
Indeed, by definition we have $\Lambda\subset\Pi$.  For the other inclusion, let $(\mathfrak{g}_1,\mathfrak{g}_2)\in \Pi$. Note that since $U$ is unitary, then $\Lambda$ is closed by the closed map theorem. Therefore we can write
\[
(\mathfrak{g}_1,\mathfrak{g}_2)= \underset{\in \Lambda}{(\mathfrak{h}_1,\mathfrak{h}_2)}+\underset{\in \Lambda^{\perp}}{(\mathfrak{f}_1,\mathfrak{f}_2)}.
\]
Since $\Lambda\subset\Pi$ we get
\[
\Pi \ni (\mathfrak{g}_1,\mathfrak{g}_2)- (\mathfrak{h}_1,\mathfrak{h}_2)=(\mathfrak{f}_1,\mathfrak{f}_2) \in \Lambda^{\perp},
\] 
thus $(\mathfrak{f}_1,\mathfrak{f}_2) \in \Lambda^{\perp}\cap \Pi$. Let us show that $(\mathfrak{f}_1,\mathfrak{f}_2)=0$. Indeed, for all $\mathfrak{g}\in \mathcal{G}$ we have
\[
\begin{split}0 & = \langle \mathfrak{f}_1, (\mathbbm{1}+U)\mathfrak{g} \rangle_{\mathcal{G}}+\langle \mathfrak{f}_2, i(\mathbbm{1}-U)\mathfrak{g} \rangle_{\mathcal{G}}\\
&= \langle U(\mathfrak{f}_1-i\mathfrak{f}_2), U\mathfrak{g} \rangle_{\mathcal{G}}+\langle \mathfrak{f}_1+i\mathfrak{f}_2, U\mathfrak{g} \rangle_{\mathcal{G}},
\end{split}
\]
which directly implies that $U(\mathfrak{f}_1-i\mathfrak{f}_2)+\mathfrak{f}_1+i\mathfrak{f}_2=0$. We also have that $(\mathfrak{f}_1,\mathfrak{f}_2)\in \Pi$, so the following  must be satisfied
\[
\begin{cases}
U(\mathfrak{f}_1-i\mathfrak{f}_2)+\mathfrak{f}_1+i\mathfrak{f}_2=0, \\
U(-\mathfrak{f}_1+i\mathfrak{f}_2)+\mathfrak{f}_1+i\mathfrak{f}_2=0.
\end{cases}
\]
Summing and subtracting these equations, and because of the injectivity of $U$, we get
\[
\begin{cases}
\mathfrak{f}_1+i\mathfrak{f}_2=0, \\
\mathfrak{f}_1-i\mathfrak{f}_2=0.
\end{cases}
\]
Then we end with $\mathfrak{f}_1=\mathfrak{f}_2=0$. Therefore, $(\mathfrak{g}_1,\mathfrak{g}_2)\in\Lambda$ and so $\Pi=\Lambda$.

Let us now assume that there exist $\mathfrak{g}_1,\mathfrak{g}_2\in \mathcal{G}$ with $\mathfrak{g}_1\neq \mathfrak{g}_2$ satisfying
\[
\Gamma\psi=\left((\mathbbm{1}+U)\mathfrak{g}_1;  \ i(\mathbbm{1}-U)\mathfrak{g}_1\right)=\left((\mathbbm{1}+U)\mathfrak{g}_2;  \ i(\mathbbm{1}-U)\mathfrak{g}_2\right).
\]
Then 
\[
\begin{cases}
(\mathbbm{1}+U)\mathfrak{g}_1=(\mathbbm{1}+U)\mathfrak{g}_2, \\
(\mathbbm{1}-U)\mathfrak{g}_1=(\mathbbm{1}-U)\mathfrak{g}_2.
\end{cases}
\]
Summing the two equations we get the contradiction $\mathfrak{g}_1=\mathfrak{g}_2$. Therefore the map $A:\psi\in\D(T_U)\mapsto \mathfrak{g}\in\mathcal{G}$ is well defined.

Let us finally prove that $\mathfrak{G}$ has dense range. 
Let $\mathfrak{g}\in \mathcal{G}$, being $\tilde{\mathcal{G}}$ dense in $\mathcal{G}$, there exists a sequence $\{\mathfrak{g}_n\}_{n}\subset \tilde{\mathcal{G}}$ such that  $\mathfrak{g}_n\to \mathfrak{g}$ in $\mathcal{G}$ for $n\to +\infty$. 
Since $U( \tilde{\mathcal{G}})\subset \tilde{\mathcal{G}}$, then $((\mathbbm{1}+U)\mathfrak{g}_n,i(\mathbbm{1}-U)\mathfrak{g}_n)\in\tilde{\mathcal{G}}\oplus\tilde{\mathcal{G}}$. 
Denoting $a_n:=(\mathbbm{1}+U)\mathfrak{g}_n=\{a_{n,j}\}_{j\in \operatorname{Conc}(V)}$ and $b_n:=(\mathbbm{1}-U)\mathfrak{g}_n=\{b_{n,j}\}_{j\in \operatorname{Conc}(V)}$, for any $n\in\mathbb{N}$ define
\[
\psi_n:=\sum_{j\in \operatorname{Conc}(V)} \left(-ia_{n,j} \varphi_j^+ + ib_{n,j} \varphi_j^- \right),
\]
where $\varphi_j^+, \varphi_j^-$ are defined in \eqref{definvarphi}. Then $\mathfrak{G}(\psi_n)=\mathfrak{g}_n$, and since $\mathfrak{g}_n\to \mathfrak{g}$ when $n\to +\infty$, this concludes the proof.

\end{proof}

We are ready now to prove the main result about the classification of the self-adjoint extensions.

\begin{proof}[Proof of \Cref{thmallselfadjointextensions}]

We start proving the symmetry of $T_U$. Let $\psi, \tilde{\psi}\in\mathcal{D}(T_U)$. Using \Cref{lemma: gammaU}, there exists $\mathfrak{g}_1,\mathfrak{g}_2\in \mathcal{G}$ such that 
\[
\Gamma \psi=((\mathbbm{1}+U)\mathfrak{g}_1; \ i(\mathbbm{1}-U)\mathfrak{g}_1)); \ \Gamma \tilde \psi=((\mathbbm{1}+U)\mathfrak{g}_2; \ i(\mathbbm{1}-U)\mathfrak{g}_2)).
\]
Thus, 
\[\begin{split}
\langle D_V^*\psi, \tilde \psi \rangle_{L^2(\Omega_V;\C^2)}-\langle \psi, D_V^* \tilde \psi \rangle_{L^2(\Omega_V;\C^2)}  = & \langle \Gamma_+\psi, \Gamma_-\tilde \psi \rangle_{\mathcal{G}}-\langle \Gamma_-\psi, \Gamma_+ \tilde \psi \rangle_{\mathcal{G}} \\
= & \langle \mathfrak{g}_1, i \mathfrak{g}_2 \rangle_{\mathcal{G}}-\langle U\mathfrak{g}_1, i U\mathfrak{g}_2 \rangle_{\mathcal{G}}+\langle U\mathfrak{g}_1, i \mathfrak{g}_2 \rangle_{\mathcal{G}}-\langle \mathfrak{g}_1, i U\mathfrak{g}_2 \rangle_{\mathcal{G}} \\
& -\langle i\mathfrak{g}_1, \mathfrak{g}_2 \rangle_{\mathcal{G}}+\langle iU\mathfrak{g}_1, U \mathfrak{g}_2 \rangle_{\mathcal{G}}+\langle iU \mathfrak{g}_1, \mathfrak{g}_2 \rangle_{\mathcal{G}} -\langle i\mathfrak{g}_1, U\mathfrak{g}_2 \rangle_{\mathcal{G}} \\
= & 0 ,
\end{split}
\]
where we have used that $U$ is unitary.
Finally, we prove that $T_U$ is self-adjoint. Let $\tilde{\psi}\in\mathcal{D}(T_U^*)\subset \D(D_V^*)$. Then for all $\psi\in\mathcal{D}(T_U)$ and using \Cref{lemma: gammaU}, we have
\[\begin{split}
 0 & = \langle \Gamma_+\tilde{\psi}, \Gamma_- \psi \rangle_{\mathcal{G}}-\langle \Gamma_-\tilde{\psi}, \Gamma_+  \psi \rangle_{\mathcal{G}} \\
& = \langle \Gamma_+\tilde{\psi},i(\mathbbm{1}-U)\mathfrak{g}_1 \rangle_{\mathcal{G}}-\langle \Gamma_-\tilde{\psi}, (\mathbbm{1}+U)\mathfrak{g}_1 \rangle_{\mathcal{G}} \\
& = \mig{\langle -i(\mathbbm{1}-U^*)\Gamma_+\tilde{\psi} -(\mathbbm{1}+U^*)\Gamma_-\tilde{\psi}, \mathfrak{g}_1\rangle_{\mathcal{G}}.}
\end{split}
\]
Since the map $A:\psi\in \mathcal{D}(T_U)\mapsto \mathfrak{g}_1\in \mathcal{G}$ has dense range by \Cref{lemma: gammaU}, we conclude that $i(\mathbbm{1}-U)\Gamma_+\tilde{\psi} =(\mathbbm{1}+U)\Gamma_-\tilde{\psi}$ and so $\tilde{\psi}\in\mathcal{D}(T_U)$. 
\end{proof}

\section{Distinguished self-adjoint extension}
\label{sec:distinguished}

In this section we are interested in giving a detailed description of the distinguished self-adjoint extension of the Dirac operator $D_V$.
We want to characterize it as the unique extension whose domain is included in the Sobolev space $H^{1/2}(\Omega_V;\C^2)$. 

By \Cref{thm:adjoint} we have that, if $\psi\in \mathcal{D}(D_V^*)$, then
\[
\psi =\psi_0+\sum_{j\in \operatorname{Conc}(V)}\left(c_j^{+}\varphi_{j}^+ +c_j^{-}  \varphi_{j}^- \right),
\] 
where the series is convergent if $\operatorname{Conc}(V)$ is infinite.
By \cite[Theorem 1.4.5.3]{grisvard} we have that $\varphi_j^+\in H^{1/2}(\Omega_V;\C^2)$ but $\varphi_j^-\notin H^{1/2}(\Omega_V;\C^2)$.
Thus, to have $\psi\in H^{1/2}(\Omega_V;\C^2)$ we must impose  $c_j^-=0$ for all $j\in \operatorname{Conc}(V)$. 
In the case that $\operatorname{Conc}(V)$ is finite we directly have that $\psi \in H^{1/2}(\Omega_V;\C^2)$. 
In the other case, this becomes a more delicate issue and it is the goal of this section.
For this reason, we assume that $\operatorname{Conc}(V)$ is infinite.
We need the following result before proving \Cref{thmH12}.

\begin{lemma}\label{lemma: inequality_sen}
Let $\Omega_V$ be a non-degenerate infinite polygon. For $j\in \operatorname{Conc}(V)$, define $\Omega_{\rho}^{(j)}:=\Omega_V\cap B_{\rho}(v^{(j)})$ and
\begin{equation}\label{eq: cj}
S_j:=\begin{cases}
1/2 & \text{ if } \ \omega_j\in(\pi,3\pi/2],\\
|\sin(\omega_j)|/2 & \text{ if } \  \omega_j\in(3\pi/2,2\pi).
\end{cases}
\end{equation}
Then, for any $x,y\in\Omega_{\rho}^{(j)}$ such that $|x-y|\leq S_j|x|$ we have 
\[
[x,y]:=\{tx+(1-t)y\in\R^2:t\in[0,1]\}\subset \Omega_{\rho}^{(j)}.
\]
\end{lemma}

\begin{proof} 
Without loss of generality (otherwise apply a rotation and a translation) we can assume  that $\Omega_\rho^{(j)}=\{(r\cos(\theta),r\sin(\theta)):\,r\in(0,\rho),\theta\in (\pi-\omega_j/2,\pi+\omega_j/2)\}$. 
Set  
\[
l_\pm:=\{(r\cos(\omega_j/2), \pm r\sin(\omega_j/2)): r\in [0,\rho]\}.
\]
Then, by construction, for any $x,y\in \Omega_\rho^{(j)}$, we have that $[x,y]\not\subset\Omega_\rho^{(j)}$ if and only if $|x-y|\geq c_x:= \max (d(x,l_-),d(x,l_+))$, 
With easy computations we have that
\[
c_x=
\begin{cases}
|x||\sin(\theta_x-\omega_j/2)| & \text{if}\ |\theta_x-\omega_j/2|<\pi/2,\\
|x| & \text{otherwise}. 
\end{cases} 
\]
Being $c_x\geq S_j|x|$, this concludes the proof.
%

\end{proof}

We are now ready to proof the main result of this section. 

\begin{proof}[Proof of \Cref{thmH12}:]
In this proof we denote by $C$ any positive constant.
Let $\{c_j\}_{j\in \operatorname{Conc}(V)}\in\ell^2(\C^2)$ and define
\[
\varphi(x):=\displaystyle\sum_{j\in \operatorname{Conc}(V)} c_j\varphi_j(x).
\]
We know that $\varphi\in L^2(\Omega_V;\C^2)$. Let us prove that, if $s$ satisfies \eqref{inf s}, then $\varphi\in H^{s}(\Omega_V;\C^2)$. In details, we want to prove that
\[
\norm{\varphi}_{H^{s}(\Omega_V;\C^2)}^2=
\norm{\varphi}_{L^2(\Omega_V;\C^2)}^2
+\underbrace{
\int_{\Omega_V}\int_{\Omega_V}
\frac{\left|\varphi(x)-\varphi(y)\right|^2}{|x-y|^{2s+2}}\,dx\, dy}_{\left[\varphi\right]^2_{H^{s}\left(\Omega_V;\C^2\right)}},
\]
is finite, that is, $[\varphi]_{H^{s}\left(\Omega_V;\C^2\right)}<+\infty$. Setting, $\Omega_\rho^{(j)}:=\Omega_V\cap B_\rho(v^{(j)})$, by construction we have that
\[
\begin{split}
[\varphi]_{H^{s}\left(\Omega_V;\C^2\right)}^2=&
\sum_{j\in \operatorname{Conc}(V)}|c_j|^2
\int_{\Omega_\rho^{(j)}}\int_{\Omega_\rho^{(j)}}
\frac{\left|\varphi_j(x)-\varphi_j(y)\right|^2}{|x-y|^{2s+2}}\,dx\, dy\\
&+\sum_{\underset{j\neq k}{j,k\in \operatorname{Conc}(V)}}
\int_{\Omega_\rho^{(j)}}\int_{\Omega_\rho^{(k)}}
\frac{\left|c_j \varphi_j(x)-c_k\varphi_k(y)\right|^2}{|x-y|^{2s+2}}\,dx\, dy
\\
=&:I+II.
\end{split}
\]
Let us prove that $I$ and $II$ are finite.
First we estimate $I$.  Without loss of generality (otherwise apply a rotation and a translation) we can assume that 
$\Omega_\rho^{(j)}=\{(r\cos(\theta),r\sin(\theta)):\,r\in[0,\rho),\theta\in (0,\omega_j)\}$. For any $j\in\operatorname{Conc}(V)$, define
\[
I_j:=
\underset{x,y\in\Omega_{\rho}^{(j)} }{\iint}
\frac{\left|\varphi_j(x)-\varphi_j(y)\right|^2}{|x-y|^{2s+2}}\,dx\, dy.  
\]
We split the integral into two parts
\[
I_j:=
\underset{ \ |x-y|\leq S_j |x|}{\underset{x,y\in\Omega_{\rho}^{(j)}}{\iint}}
\frac{\left|\varphi_j(x)-\varphi_j(y)\right|^2}{|x-y|^{2s+2}}\,dx\, dy  \ +
\underset{ \ |x-y|\geq S_j |x| }{\underset{x,y\in\Omega_{\rho}^{(j)}}{\iint}}
\frac{\left|\varphi_j(x)-\varphi_j(y)\right|^2}{|x-y|^{2s+2}}\,dx\, dy  =: I_{j,1}+I_{j,2},   
\]
where $S_j>0$ is defined in \eqref{eq: cj}. We estimate $I_{j,1}$ first. By \Cref{lemma: inequality_sen}, we can apply the mean value theorem, that combined with the Cauchy-Schwartz inequality, it gives us that there exists $\xi_{x,y}\in [x,y]$ such that 
\begin{equation}\label{ineq: nabla1}
|\varphi_j(x)-\varphi_j(y)|\leq |\nabla \varphi_j(\xi_{x,y})||x-y|.
\end{equation}
We also have that for  $\xi=(r\cos(\theta),r\sin(\theta))\in\Omega_\rho^{(j)}$, in polar coordinates
\[
\begin{split}
|\nabla \varphi_j(\xi)|^2 & =
 \left|
 \partial_r\left(r^{\lambda_j-1/2}\phi\left(\frac{r}{\rho}\right)\Phi_j(\theta)\right)\right|^2+ \frac{1}{r^2}\left|\partial_{\theta}\left(r^{\lambda_j-1/2}\phi\left(\frac{r}{\rho}\right)\Phi_j(\theta)\right)\right|^2\\
& = \frac{r^{2\lambda_j-3}}{\omega_j}\left(2(\lambda_j-1/2)^2\phi \left(\frac{r}{\rho}\right)^2+\frac{r^2}{\rho^2}\phi'\left(\frac{r}{\rho}\right)^2+2(\lambda_j-1/2)\frac{r}{\rho}\phi \left(\frac{r}{\rho}\right)\phi'\left(\frac{r}{\rho}\right)\right) \\
& \leq C r^{2\lambda_j-3},
\end{split}
\]
where the last inequality holds by definition of $\phi$. 
Moreover, as $|\xi_{x,y}|\geq \left||x|-|x-y|\right|\geq (1-S_j)|x|$, we get by \eqref{ineq: nabla1} that
\[
|\varphi_j(x)-\varphi_j(y)|^2\leq C(1-S_j)^{2\lambda_j-3}|x-y|^2 |x|^{2\lambda_j-3}.
\]
Finally, by Fubini's theorem,
\[
\begin{split}
I_{j,1} & \leq C(1-S_j)^{2\lambda_j-3} \underset{ \ |x-y|\leq S_j |x|}{\underset{x,y\in\Omega_{\rho}^{(j)}}{\iint}}|x|^{2\lambda_j-3}|x-y|^{-2s}dx dy \\
& \leq C(1-S_j)^{2\lambda_j-3}  \underset{x\in\Omega_{\rho}^{(j)}}{\int}|x|^{2\lambda_j-3} \left( \underset{|x-y|\leq S_j|x|}{\int} |x-y|^{-2s}dy\right)dx \\
& \leq CS_j^{2-2s}(1-S_j)^{2\lambda_j-3}  \underset{x\in\Omega_{\rho}^{(j)}}{\int}|x|^{2\lambda_j-2s-1} dx \\
& = CS_j^{2-2s}(1-S_j)^{2\lambda_j-3}  \rho^{2\lambda_j-2s+1},\end{split}
\] 
where, in the last equality, we have used that $2\lambda_j-2s-1>-2$ thanks to \eqref{inf s}.

Now we estimate $I_{j,2}$. Since $|\varphi_j(x)-\varphi_j(y)|^2\leq 2(|\varphi_j(x)|^2+|\varphi_j(y)|^2)$,  and by symmetry
\[
\begin{split}
I_{j,2}  & \leq 4 \underset{ |x-y|\geq S_j |x|}{\underset{x,y\in\Omega_{\rho}^{(j)}}{\iint}}|\varphi(x)|^2|x-y|^{-2s-2}dx dy \\
& \leq \frac{4}{\omega_j} \underset{x\in\Omega_{\rho}^{(j)}}{\int}|x|^{2\lambda_j-1} \left( \underset{|x-y|\geq S_j|x|}{\int} |x-y|^{-2s-2}dy\right)dx \\
&=  \frac{8\pi}{\omega_jS_j^{2s}} \underset{x\in\Omega_{\rho}^{(j)}}{\int}|x|^{2\lambda_j-2s-1} dx \leq \frac{C}{S_j^{2s}} \rho^{2\lambda_j-2s+1},
\end{split}
\] 
where, in the last equality, we have used that $2\lambda_j-2s-1>-2$.

Therefore, we can conclude that
\[
I=\sum_{j\in \operatorname{Conc}(V)}|c_j|^2(I_{j,1}+I_{j,2})\leq\displaystyle C \sum_{j\in \operatorname{Conc}(V)} |c_j|^2 \left( S_j^{2-2s}(1-S_j)^{2\lambda_j-3}+\frac{1}{S_j^{2s}}\right).
\]
Since $2\pi-\omega_j>\epsilon$ for all $j\in\text{Conc}(V)$, then there exists $\delta>0$ such that $\delta\leq S_j\leq 1/2$. Thus, the term within parentheses is bounded. Thanks to this and since $\{c_j\}_{j\in \operatorname{Conc}(V)}\in \ell^2(\C^2)$, $I$ is finite.
 
Let us estimate $II$. 
Since $|c_j\varphi_j(x)-c_k\varphi_k(y)|^2\leq 2(|c_j\varphi_j(x)|^2+|c_k\varphi_j(y)|^2)$, and by symmetry
\[
II\leq 4 \sum_{\underset{j\neq k}{j,k\in \operatorname{Conc}(V)}}
\int_{\Omega_\rho^{(j)}}\int_{\Omega_\rho^{(k)}}
\frac{|c_j \varphi_j(x)|^2}{|x-y|^{2s+2}}\,dx\, dy.
\]
By definition of non-degenerate infinite polygon, for $j\neq k$ and $x\in\Omega_\rho^{(j)}$, $y\in \Omega_\rho^{(k)}$, we have that
\[
|x-y|\geq |v^{(j)}-v^{(k)}| - |x-v^{(j)}|-|y-v^{(k)}|\geq 3\rho-2\rho\geq\rho.
\]
Thanks to this and by Fubini's Theorem,
\[
\begin{split}
II&
\leq 
4
\sum_{j\in \operatorname{Conc}(V)}
\int_{\Omega_\rho^{(j)}} |c_j \varphi_j(x)|^2
\left(\sum_{\underset{k\neq j}{k\in \operatorname{Conc}(V)}}\int_{\Omega_\rho^{(k)}}
|x-y|^{-2s-2}dy\right)dx \\
& \leq 
4
\sum_{j\in \operatorname{Conc}(V)}
\int_{\Omega_\rho^{(j)}} |c_j \varphi_j(x)|^2
\left(\int_{|z|>\rho}
|z|^{-2s-2}\,dz\right)\,dx\\
&\leq
 C\sum_{j\in \operatorname{Conc}(V)}|c_j|^2
\int_{\Omega_\rho^{(j)}}  |x|^{2\lambda_j-2s-1} dx \\
&=
C \sum_{j\in \operatorname{Conc}(V)} \frac{\rho^{2\lambda_j-2s+1}|c_j|^2 }{(2\lambda_j-2s+1)}.
\end{split}
\] 
Here $\{c_j\}_{j\in \operatorname{Conc}(V)}\in \ell^2(\C^2)$ and $s$ verifies \eqref{inf s}. Therefore $II$ is finite and this concludes the proof.

\end{proof}

\section{Spectral properties and proof of \texorpdfstring{\Cref{thm:spectrum}}{Theorem 1.8}}
\label{sec:spectrum}
This section is dedicated to the proof of \Cref{thm:spectrum} and \Cref{rem:bound-spectrum}, building upon the same strategy employed in the proof of \cite[Proposition 1.12]{LeTeOurBon18}.
We confine our investigation of spectral properties to $D_{sa}$ defined in \eqref{def: Dsa}. 

\begin{proof}[Proof of \Cref{thm:spectrum}]
Let us prove that $(-\infty,-m]\cup[m,+\infty)
\subseteq \sigma(D_{sa})$.
The main idea is to construct a Weyl sequence for the operator $D_{sa}$: for a fixed $|\mu|\geq m$, and $x=(x_1,x_2)\in \Omega_V$ define  
\[
\psi(x):=
\begin{pmatrix}
\sqrt{|\mu|+\operatorname{sign}(\mu)m} \\\\
i\operatorname{sign}(\mu)\sqrt{|\mu|-\operatorname{sign}(\mu)m}
\end{pmatrix}
\frac{e^{i\sqrt{\mu^2-m^2}\,x_2}}{\sqrt{4\pi |\mu|}}.
\]
We remark that $(\Dirac_m-\mu)\psi=0$. Moreover,
thanks to \eqref{eq:bdd-condition-spectrum}, it exists a sequence $\{x_n\}_{n\in \N}\subseteq \Omega_V$ such that $|x_n|\to +\infty$ for $n\to +\infty$ and $\{ x\in\R^2:|x-x_n|<n \}\subset \Omega_V$. Then for $n\in \mathbb{N}$  define
\[
\psi_n(x)=\psi(x)\phi\left(\frac{|x-x_n|}{n}\right),
\]
being $\phi$ the cut-off function defined in \eqref{cutoff}. Then, by construction $\psi_n\in \D(D_V)\subseteq\D(D_{sa})$.
Let us prove that this is a Weyl sequence. Indeed we can easily compute
\[
\norm{\psi_n}_{L^2(\Omega_V;\mathbb{C}^2)}^2=
n^2
\int_0^1 r\phi(r)^2\,dr
\quad
\text{and}
\quad
\norm{(D_{sa}-\mu)\psi_n}_{L^2(\Omega_V;\mathbb{C}^2)}^2
=
\int_0^1 r\phi'(r)^2\,dr.
\]
Thus
\[
\lim_{n\to+\infty} \frac{\norm{(D_{sa}-\mu)\psi_n}_{L^2(\Omega_V;\mathbb{C}^2)}^2}{\norm{\psi_n}_{L^2(\Omega_V;\mathbb{C}^2)}^2}=0,
\]
and therefore $\mu\in \sigma(D_{sa})$. 

We now discuss the reverse inclusion. Let $\psi\in\mathcal{D}(D_{sa})$, then
\begin{equation}\label{eq:square}
\begin{split}
\norm{D_{sa}\psi}_{L^2(\Omega_V; \mathbb{C}^2)}^2= &\ \norm{-i\s \psi}_{L^2(\Omega_V; \mathbb{C}^2)}^2+m^2\norm{\psi}_{L^2(\Omega_V; \mathbb{C}^2)}\\
& +2m\operatorname{Re}\langle -i\s \psi,\sigma_3 \psi \rangle_{L^2(\Omega_V; \mathbb{C}^2)}.
\end{split}
\end{equation}
We state that 
\begin{equation}\label{eq:2Re-commutator}
2\operatorname{Re}\langle 
-i\s \psi,\sigma_3 \psi 
\rangle_{L^2(\Omega_V; \mathbb{C}^2)}
= \norm{\psi}_{L^2(\partial\Omega;\C^2)}^2.
\end{equation}
Since $m\geq 0$, we can conclude that
$
\norm{D_{sa}\psi}_{L^2(\Omega_V; \mathbb{C}^2)}^2
\geq 
m^2\norm{\psi}_{L^2(\Omega_V; \mathbb{C}^2)}.
$
Therefore, $\sigma(D_{sa}^2)\subset [m^2,+\infty)$ 
and so 
$
\sigma(D_{sa})\subset (-\infty,-m]\cup [m,+\infty)$. 

Let us prove \eqref{eq:2Re-commutator}.
By \Cref{thmH12},
\[
\psi =\psi_0 +\sum_{j\in \operatorname{Conc}(V)}c_j^+\varphi_j^+,
\]
where $\psi_0\in \mathcal{D}(D_V)$, $\varphi_j^+$ is defined in \eqref{definvarphi} and the series is convergent in the norm of $D_{sa}$ if $\operatorname{Conc}(V)$ is infinite.
Since the functions $\varphi_j^+$ are compactly supported and $\operatorname{supp}(\varphi_j^+)\cap \operatorname{supp}(\varphi_k^+)=\emptyset$ for $j\neq k$, we have
\[
\begin{split}
2\operatorname{Re}\langle 
-i\s \psi,\sigma_3 \psi
\rangle_{L^2(\Omega_V; \mathbb{C}^2)} 
=&\
2\operatorname{Re}
\langle
-i\s \psi_0,\sigma_3\psi_0 
\rangle_{L^2(\Omega_V; \mathbb{C}^2)}
\\
&+
\sum_{j\in \operatorname{Conc}(V)}|c_j|^2
2\operatorname{Re}\left\langle
 -i\s \varphi_j^+,
 \sigma_3\varphi_j^+ 
\right\rangle_{L^2(\Omega_V; \mathbb{C}^2)}
\\
&+
\sum_{j\in \operatorname{Conc}(V)}
2\operatorname{Re}\left\langle 
-i\s \psi_0,\sigma_3 c_j^+\varphi_j^+ 
\right\rangle_{L^2(\Omega_V; \mathbb{C}^2)}
\\
&+
\sum_{j\in \operatorname{Conc}(V)}
2\operatorname{Re}
\left\langle
 -ic_j^+\s\varphi_j^+,\sigma_3\psi_0 
 \right\rangle_{L^2(\Omega_V; \mathbb{C}^2)}.
\end{split}
\]

Let us estimate the four terms separately.
About the first one, since $\psi_0\in \mathcal{D}(D_V)$ we can integrate by parts by \Cref{prop:sigma.nabla=nabla}
\[
2\text{Re}\langle -i\s \psi_0, \sigma_3 \psi_0 \rangle_{L^2(\Omega_V;\mathbb{C}^2)}=\norm{\psi_0}^2_{L^2(\partial \Omega_V;\mathbb{C}^2)}.
\]
The second term can be estimate explicitly. Indeed, thanks to \eqref{proof:diracpolar} and integrating by parts we get
\[
\begin{split}
2\operatorname{Re}\langle -i\s \varphi_j^+, \sigma_3 \varphi_j^+ \rangle_{L^2(\Omega_V;\mathbb{C}^2)}&=
-\frac{2}{\pi}
\int_0^\rho 
r^{2\lambda_j}
\left(\phi^2\left(\frac{r}{\rho}\right)\right)'\,dr
=
\frac{2}{\omega}\int_0^\rho
r^{2\lambda_j-1} \phi^2 \left(\frac{r}{\rho }\right)\,dr
\\
&
=
\norm{\varphi_j^+}^2_{L^2(\partial \Omega_V;\mathbb{C}^2)}.
\end{split}
\]
About the fourth terms, since $\psi_0\in H^1(\Omega_V;\C^2)$ we can integrate by parts. Using the anti-commutation property of the Pauli matrices and the boundary conditions we get
\[
\begin{split}
2\operatorname{Re}
\left\langle
 -ic_j^+\s\varphi_j^+,\sigma_3\psi_0 
 \right\rangle_{L^2(\Omega_V; \mathbb{C}^2)}
=&
-2\operatorname{Re}
\left\langle 
c_j^+\sigma_3\varphi_j^+,-i\s\psi_0 
\right\rangle_{L^2(\Omega_V; \mathbb{C}^2)}\\
&+2
\operatorname{Re}\left\langle 
 c_j^+\varphi_j^+, \psi_0
\right\rangle_{L^2(\partial\Omega_V; \mathbb{C}^2)}.
\end{split}
\]
Summing the four terms we get that \eqref{eq:2Re-commutator} is proved and this concludes the proof.
\end{proof}
\begin{remark}
As a consequence of \eqref{eq:square} and \eqref{eq:2Re-commutator}, we have that for functions $\psi\in \D(D_{sa})$ it is possible to define a boundary trace in $L^2(\partial\Omega_V;\C^2)$. This property is not trivial since, in general, functions in $H^{1/2}$ do not have a $L^2$-regular boundary trace.
\end{remark}

To conclude the section it remains to prove \Cref{rem:bound-spectrum}.

\begin{proof}[Proof of \Cref{rem:bound-spectrum}]
The domain $\Omega_V=(-1,1)\times(0,+\infty)$ has no concave corners and so the operator $D_V$ defined in \eqref{defdomainDV} is self-adjoint. Let $\psi\in  \D(D_V)$ and let us compute 
\[
\begin{split}
\norm{
D_V\psi
}_{L^2(\Omega_V;\C^2)}^2=&
\int_0^\infty\int_{-1}^1 
\left|(-i\sigma_1\partial_{x_1}+m\sigma_3)\psi\right|^2
\,dx_1\, dx_2
\\
&+
\int_0^\infty\int_{-1}^1 
\left|-i\sigma_2\partial_{x_2}\psi\right|^2
\,dx_1\,dx_2
\\
&+
2\operatorname{Re}
\int_0^\infty\int_{-1}^1 
\left\langle
-i\sigma_1\partial_{x_1}\psi,-i\sigma_2\partial_{x_2}\psi\right\rangle_{\C^2}
\,dx_1\,dx_2\\
&+
2m\operatorname{Re}
\int_0^\infty\int_{-1}^1 
\left\langle
\sigma_3\psi,-i\sigma_2\partial_{x_2}\psi\right\rangle_{\C^2}
\,dx_1\,dx_2\\
=&:I+II+III+IV.
\end{split}
\]
Let us estimate these terms separately.

About the first one, for any fixed $x_2\in(0,+\infty)$ the function $\psi(\cdot,x_2)$ belongs to the domain of the \emph{transverse} Dirac operator $D_{tr}$ defined as
\[
\D(D_{tr}):=\left\{ u\in H^1((-1,1);\C^2): u_2(\pm 1)=\pm i u_1(\pm1)\right\}.
\]
This operator has been introduced in an equivalent way in \cite{BorBrKrOB22}. Thanks to \cite[Proposition 10]{BorBrKrOB22} we have
\[
\norm{D_{tr}\psi(\cdot,x_2)}_{L^2((-1,1);\C^2)}^2
\geq 
\left(m^2+E_1(m)\right)
\norm{\psi(\cdot,x_2)}_{L^2((-1,1);\C^2)}^2,
\]
and thus $I\geq \left(m^2+E_1(m)\right)\norm{\psi}_{L^2(\Omega_V;\C^2)}^2$.

About the second one, we have $II\geq 0$.
About the third term, repeating the proof of \cite[(iii) of Lemma 2.1]{PizzVDB21} and since the curvature of $\partial\Omega_V$ is zero, we get that $III=0$.

About the last term, thanks to Fubini's theorem and reasoning as in the proof of \eqref{eq:2Re-commutator}, we have that
\[
IV=
m
\int_{-1}^{1}
|\psi(x_1,0)|^2
\,dx_1\geq 0.
\]
Combining these facts, we have that $\norm{D_V\psi}_{L^2(\Omega_V; \mathbb{C}^2)}^2\geq \left(m^2+E_1(m)\right)\norm{\psi}_{L^2(\Omega_V; \mathbb{C}^2)}^2$  and thus
\[
\sigma(D_V)\subseteq\left(-\infty,-\sqrt{m^2+E_1(m)}\right]\cup\left[\sqrt{m^2+E_1(m)},+\infty\right).
\]
To prove the other inclusion, set for simplicity $M:=\sqrt{m^2+E_1(m)}$, and for $t\in(-1,1)$ define 
\[
\begin{split}
\Phi^+(t)&:=
C
\begin{pmatrix}
 \left(M+m\right) \sin \left((t+1) \sqrt{E_1(m)}\right)+\sqrt{E_1(m)} \cos \left((t+1) \sqrt{E_1(m)}\right) \\
 -i \left(\left(m-M\right) \sin \left((t+1) \sqrt{E_1(m)}\right)+\sqrt{E_1(m)} \cos \left((t+1) \sqrt{E_1(m)}\right)\right) 
\end{pmatrix},
\\
\Phi^-(t)&:=
C
\begin{pmatrix}
 i \left(\left(m-M\right) \sin \left((t+1) \sqrt{E_1(m)}\right)+\sqrt{E_1(m)} \cos \left((t+1) \sqrt{E_1(m)}\right)\right) \\
 \left(M+m\right) \sin \left((t+1) \sqrt{E_1(m)}\right)+\sqrt{E_1(m)} \cos \left((t+1) \sqrt{E_1(m)}\right) 
\end{pmatrix},
\end{split}
\] 
being $C>0$ a normalization constant in $L^2((-1,1);\C^2)$. 
Notice that $\Phi^\pm\in \D(D_{tr})$ and 
\begin{equation}\label{eq:eigen-transv-oper}
D_{tr}\Phi^{\pm}=\pm M \Phi^{\pm}.
\end{equation}
For $|\mu|\geq M$ and for $x=(x_1,x_2)\in \Omega_V$ set
\[
\psi(x):=
\left(
\sqrt{|\mu|+\operatorname{sign}(\mu)M}
\Phi^+(x_1)
+
i\operatorname{sign}(\mu)\sqrt{|\mu|-\operatorname{sign}(\mu)M}
\Phi^-(x_1)
\right)
\frac{e^{i\sqrt{\mu^2-M^2}\,x_2}}{2|\mu|}.
\]
Thanks to \eqref{eq:eigen-transv-oper} and using the fact that $-i\sigma_2 \Phi^+=\Phi^-$, we have $(\Dirac_m-\mu)\psi=0$. For $n>0$, define
\[
\psi_n(x_1,x_2):=\psi(x_1,x_2)\cdot\phi\left(\frac{|x_1-n|}{n}\right)
\in \D(D_V),
\]
being $\phi$ the cut-off function defined in \eqref{cutoff}.
This is a Weyl sequence. Indeed, 
thanks to the fact that $\Phi^+$ and $\Phi^-$ are orthonormal functions in $L^2((-1,1);\C^2)$, we can easily compute
\[
\norm{\psi_n}_{L^2(\Omega_V;\mathbb{C}^2)}^2=
n
\norm{\phi}_{L^2(0,1)}^2
\quad
\text{and}
\quad
\norm{(D_V-\mu)\psi_n}_{L^2(\Omega_V;\mathbb{C}^2)}^2
=
\norm{\phi'}_{L^2(0,1)}^2.
\]
Thus
\[
\lim_{n\to+\infty} \frac{\norm{(D_V-\mu)\psi_n}_{L^2(\Omega_V;\mathbb{C}^2)}^2}{\norm{\psi_n}_{L^2(\Omega_V;\mathbb{C}^2)}^2}=0,
\]
and therefore $(-\infty,-M]\cup[M,+\infty)\subset\sigma(D_V)$.
This concludes the proof.
\end{proof}

\section*{Acknowledgements}
We would like to thank Luis Vega for the enlightening discussions.
M.~C.~is supported by the Basque Government through the BERC 2022-2025 program and by the Ministry of Science and Innovation: BCAM Severo Ochoa accreditation CEX2021-001142-S / MICIN / AEI / 10.13039/501100011033.
F.~P.~is supported by the project  PID2021-123034NB-I00 funded by funded by MCIN/ AEI/10.13039/501100011033/ FEDER, UE.
We declare that we do not have  conflict of interests related to this publication.

\appendix
\section{Polar decomposition of the infinite mass Dirac operator on a sector}
\label{sec:polar}
The goal of this section is to give a polar decomposition of the localized Dirac operator $D_\rho^{(j)}$ defined in \eqref{eq:def-D-rho-j}. Setting for convenience $\Omega_\rho^{(j)}:=\Omega_V\cap B_\rho(v^{(j)})$ and after applying a rotation and translation, we can assume without loss of generality that
\begin{equation}\label{eq:def-Omega_rho}
\Omega_\rho^{(j)}=\{(r\cos(\theta),r\sin(\theta)):r\in(0,\rho),\,\theta\in(0,\omega_j)\}.
\end{equation}
Since $j\in \operatorname{Conc}(V)$ is fixed, in this section we simplify the notation and we omit it.

The subject of the polar decomposition of the infinite mass Dirac operator on a sector is very well known, so we give a brief presentation of it here. For full details the reader may see \cite{LeTeOurBon18,PizzVDB21,CGP22,CL20,FL23} or \cite[Section 4.6]{Thaller92} for the analogous three dimensional reduction.

We use the standard notation for polar coordinates: for $x=(x_1,x_2)\in\R^2\setminus\{0\}$, 
\[
	\begin{array}{l}
		x_1=r\cos(\theta),\\
		x_2= r\sin(\theta),
	\end{array}
	\quad\text{being}\quad
	\begin{array}{l}
		r:=\sqrt{x_1^2+x_2^2}\in (0,+\infty),\\
		\theta:=\frac{\pi}{2}\cdot(1-\operatorname{sign}(x_2))+\arccos\left(x_1/r \right)\in[0,2\pi).
	\end{array}
\]
For all $\psi \in L^2(\Omega_\rho;\C)$, let $\varphi : (0,\rho) \times (0,\omega) \to \C$ be defined as follows
\[
	\varphi(r,\theta):= \sqrt{r} \, \psi \left(r\cos(\theta),r\sin(\theta)\right),
\]
The map $\psi \mapsto \varphi$ is a unitary map $L^2(\Omega_\rho; \C) \to L^2((0,\rho);\C)\otimes L^2((0,\omega);\C)$, since
\[
	\int_{\Omega_\rho} |\psi(x)|^2\,\ud x\,=\,
		\int_0^\omega \int_0^{+\infty} |\varphi(r,\theta)|^2 \, dr\,d\theta.
\]
Repeating this reasoning for every component of the wave-function, we obtain the decomposition
\begin{equation}\label{eq:equivL2}
	L^2(\Omega_\rho;\C^2)\,\cong\,
	L^2((0,\rho);\C^2)\otimes L^2((0,\omega);\C^2),
\end{equation}
where `` $\cong$ ''means unitarly equivalent. 

It is useful to express the Dirac operators in polar coordinates: setting
\[
	\boldsymbol{e_r}:=(\cos(\theta),\sin(\theta))=\frac{x}{r},\quad
	\boldsymbol{e_\theta}:=(-\sin(\theta),\cos(\theta))=\frac{\partial \boldsymbol{e_r}}{\partial\theta},
\]
we abbreviate  
\[
	\partial_r=\boldsymbol{e_r}\cdot{\nabla}
	\quad\text{and}\quad
	\partial_\theta=\boldsymbol{e_\theta}\cdot{\nabla}.
\]
By means of elementary computations, it is easy to see that
\[
	\boldsymbol{\sigma} \cdot \boldsymbol{e_r} \;=\; \begin{pmatrix}
		0 & e^{-i \theta} \\
		e^{i \theta} & 0
	\end{pmatrix}.
\]
and that the identity $\boldsymbol{\sigma}\cdot\boldsymbol{e_\theta} = i \boldsymbol{\sigma}\cdot\boldsymbol{e_r} \sigma_3$ holds. We obtain
\[
	-i\boldsymbol{\sigma}\cdot\nabla= 
		-i \boldsymbol{\sigma}\cdot \left(\boldsymbol{e_r} \partial_r + \frac{1}{r}\boldsymbol{e_\theta}\partial_\theta \right)
	 = -i \boldsymbol{\sigma}\cdot\boldsymbol{e_r} \left(\partial_r  +\frac{1}{2r}-\frac{K_\omega}{r}  \right),
\]
where $K_\omega$ is the \emph{spin-orbit operator}
\[
K_\omega := \frac{1}{2}\mathbbm{1} - i \sigma_3 \partial_\theta.
\]
To get an appropriate decomposition of $L^2((0,\omega);\C^2)$, we recall \cite[Lemma 2.4]{PizzVDB21},  about the properties of $K_\omega$.
\begin{proposition}[Properties of the spin-orbit operator]\label{prop:properties.spin-orbit}
Let $\omega\in(0,2\pi]$, $\Omega_\rho$ as in \eqref{eq:def-Omega_rho} and for any $k\in\mathbb{N}$ set
\begin{equation}\label{eq:def-lambda_k}
\tau_k:=\frac{(2k+1)\pi}{2\omega},
\end{equation}
and
\begin{equation}\label{eq:def-f_k}
	f_k^+ (\theta) \; := \; \frac{1}{\sqrt{2\omega}} \begin{pmatrix}
				e^{i(\tau_k-\frac{1}{2}) \theta} \\ e^{-i(\tau_k-\frac{1}{2}) \theta}
						\end{pmatrix}, \quad 
	f_k^- (\theta) \; := \; \frac{-i}{\sqrt{2\omega}} \begin{pmatrix}
				e^{-i(\tau_k+\frac{1}{2}) \theta} \\ e^{i(\tau_k+\frac{1}{2}) \theta}
						\end{pmatrix},
		\quad \text{ for }\theta \in (0,\omega).
		\end{equation}
The spin-orbit operator with infinite mass boundary conditions
\[
	\begin{split}
		& K_\omega \;:=\; \frac{1}{2} \mathbbm{1}- i \sigma_3 \partial_\theta \, , \\
		& \mathcal{D}(K_\omega) \;:=\; 
			\big\{ 
				f=(f_1,f_2) \in H^1((0,\omega), \mathbb{C}^2) \, : \, 
				f_2(\omega)=-e^{i \omega} f_1(\omega) \, , \,
				f_2(0)=f_1(0) 
			\big\}\, ,
			\end{split}
\]
has the following properties:
	\begin{enumerate}[label=$({\roman*})$]
		\item $K_\omega$ is self-adjoint and it has a compact resolvent;
		\item\label{item:ONB}
		$\{f^+_k, f^-_k\}_{k \in \mathbb{N}}$  is an orthonormal basis of eigenfunctions of $L^2((0,\omega);\mathbb{C}^2)$ with eigenvalues 
	$\{ \tau_k,-\tau_k\}_{k \in \mathbb{N}}$; 
		\item $-i(\boldsymbol{\sigma} \cdot \boldsymbol{e_r} ) f^\pm_k = \pm  f^\mp_k $ for all $k \in \N$.
	\end{enumerate}
\end{proposition}

In the following proposition we finally decompose the space $L^2(\Omega_\rho;\C^2)$ in partial wave subspaces.
\begin{proposition}[Polar decomposition]\label{prop:decomposition.partialwave}
	Let $\omega\in (0,2\pi]$, $\Omega_\rho$ be defined as in \eqref{eq:def-Omega_rho} and for all $k\in \mathbb{N}$ let $\tau_k^{\pm}$ as in \eqref{eq:def-lambda_k} and $f_k^{\pm}$ as in \eqref{eq:def-f_k}.
	Then, 
for any $\psi\in L^2(\Omega_\rho;\C^2)$ there exists $\{(u^+_k,u^-_k)\}_{k\in\N}\in L^2(0,+\infty)\oplus L^2(0,+\infty)$ such that
\[
	\begin{gathered}
	\psi(r,\theta)=\frac{1}{\sqrt{r}}\sum_{k\in \N}
	\left[u_k^+(r) f_k^+(\theta)+ u_k^-(r) f_k^-(\theta)\right]
	\quad
	\text{ for } r\in (0,+\infty), \theta \in (0,\omega),\\
	\Vert \psi\Vert_{L^2(\Omega_\rho;\C^2)}^2 = 
	\sum_{k \in \N} 
	\left[ \Vert u_k^+\Vert_{L^2(0,\rho)}^2 +
	\Vert u_k^-\Vert_{L^2(0,\rho)}^2 \right].
	\end{gathered}
\]

	\end{proposition}
	\begin{proof}
	The proof is immediate from \eqref{eq:equivL2} and \ref{item:ONB} of \Cref{prop:properties.spin-orbit}.
	\end{proof}
Thanks to \Cref{prop:decomposition.partialwave}, it is possible to decompose the Dirac operator $D_\rho$ as the direct sum of the one dimensional Dirac operators on the half-line. For any $k\in\mathbb{N}$ define the differential expression
\begin{equation}\label{eq:def-slashed-dk}
\slashed{\mathfrak{d}_k}:=
\begin{pmatrix}
0&-\partial_r-\frac{\tau_k}{r}\\
\partial_r-\frac{\tau_k}{r} & 0
\end{pmatrix}.
\end{equation}
\begin{proposition}\label{prop:UnitaryTransformation}
Let $\omega \in (0,2\pi]$, $\Omega_\rho$ be defined as in \eqref{eq:def-D-rho-j} and $D_\rho$ defined as in \eqref{eq:def-D-rho-j}.
For all $k \in \N$, define
\[
\D(d_k):=  H^1_0\big((0,\rho);\C^2\big),\quad d_k u= \slashed{\mathfrak{d}_k}u,
\]
be defined as in \eqref{eq:def-slashed-dk}. Then, if $m=0$,
\[
		 D_\rho \;\cong\; \bigoplus_{k \in \N} d_k,
\]
	where `` $\cong$ '' means that the operators are unitarily equivalent. 
In detail, for $\psi \in L^2(\Omega_\rho;\C^2)$ there exists $\{u_k=(u^+_k,u^-_k)\}_{k\in\N}\subset L^2\left((0,\rho);\C^2\right)$ such that  
	\begin{equation}
	\label{eq:equivalence.psi.radial}
	\psi \in \D(D_\rho) \iff u_k\in \D(d_k) = H^1_0\big((0,\rho);\C^2\big)
	\
	\text{ for all }k \in \N.
	\end{equation}
And for $r\in(0,\rho)$ and $\theta \in (0,\omega)$,
\begin{equation}\label{eq:compute.HV}
D_\rho \psi (r,\theta) = \frac{1}{\sqrt{r}}\sum_{k\in \N}
\left[
\left(-\partial_r-\frac{\tau_k}{r}\right)
u_k^-(r)f_k^+(\theta)+
\left(\partial_r-\frac{\tau_k}{r}\right)
u_k^+(r)f_k^-(\theta)
\right].
\end{equation}
\end{proposition}

\begin{proof}
\eqref{eq:compute.HV} has been already proved in \cite{CGP22}, thus we only need to prove the equivalence in \eqref{eq:equivalence.psi.radial}. 

Thanks to \Cref{prop:sigma.nabla=nabla}, one and since $f_k^\pm \in C_c^{\infty}([0,\omega];\C^2)$, if $u_k\in H^1_0\left((0,\rho);\C^2\right)$ for any $k\in \mathbb{N}$, then $\psi\in \D(D_\rho)$.

Let us prove the other equivalence.
Set $O_\rho:=\left(\overline{\Omega_\rho}\cap B_\rho\right)\setminus\{0\}$ and set
\[
\D_\rho:=
\left\{\psi\in C_c^{\infty}(O_\rho;\C^2):
\ -i \sigma_3  \boldsymbol{\sigma} \cdot \textbf{n}  u= u\ \text{on}\ \partial O_\rho\right\}.
\] 
If $\psi\in \D_\rho$, using \eqref{eq:equivalence.psi.radial} and since $f_k^\pm \in C_c^{\infty}([0,\omega];\C^2)$, we get $u_k\in C^\infty_c\left((0,\rho);\C^2\right)$.
We claim that  
\begin{equation}\label{eq:density}
\overline{\D_\rho}^{D_\rho}=\mathcal{D}(D_\rho)
\quad\text{and}\quad
\overline{C^\infty_c\left((0,\rho);\C^2\right)}^{d_k}=H^1_0\left((0,\rho);\C^2\right).
\end{equation}
This combined with \Cref{prop:decomposition.partialwave} and with \eqref{eq:compute.HV} concludes the proof.
About the first equality of \eqref{eq:density}, from \eqref{eq:sigma.nabla=nabla} we deduce that the $D_\rho$ norm is equivalent to the $H^1$-norm on $\D_\rho$.
The second equality of \eqref{eq:density} descends from \cite[Lemma A.1]{FL23} and the fact that
$\tau_k\neq 1/2$ for any $k\in\mathbb{N}$.
\end{proof}
We are now ready to prove \eqref{eq:estimate-bessel-zero}.
\begin{lemma}\label{lemmasigmav}
Let $\omega \in (0,2\pi)$, $\Omega_\rho$ be defined as in \eqref{eq:def-D-rho-j} and $D_\rho$ defined as in \eqref{eq:def-D-rho-j}. Then,  for any $\psi\in\mathcal{D}(D_{\rho})$, 
\[
\norm{\s \psi}_{L^2(\Omega_{\rho};\C^2)}\geq \frac{2}{\rho}\norm{\psi}_{L^2(\Omega_{\rho};\C^2)}.
\]
\end{lemma}

\begin{proof}
Let $\psi \in\D(D_\rho)$, such that $\norm{\psi}_{L^2(\Omega_\rho;\C^2)}=1$. By \Cref{prop:UnitaryTransformation} we have that there exists $\{u_k=(u_k^+,u_k^-)\}_{k\in\mathbb{N}}\subset H^1_0\left((0,\rho);\C^2\right)$ such that
\begin{equation}
\label{eq:dec-in-radial}
\norm{D_{\rho}\psi}_{L^2(\Omega_\rho;\C^2)}^2=
\sum_{k=0}^{\infty}
\int_{0}^{\rho}\left|\left(\partial_r-\frac{\tau_k}{r}\right)u_k^+(r)\right|^2 dr
+
\int_{0}^{\rho}\left|\left(\partial_r+\frac{\tau_k}{r}\right)u_k^-(r)\right|^2 dr
\end{equation}
Thanks to this, we can reduce the proof to the solution of an eigenvalue equation.
Indeed, for any $\tau\in\R\setminus \{1/2\}$ define the following quadratic form
\[
\operatorname{dom}(q_\tau):=H^1_0\left((0,\rho);\C\right)\quad\text{and}\quad
q_\tau[u]:=\int_0^\rho \left|\left(\partial_r-\frac{\tau}{r}\right)u(r)\right|^2 dr.
\]
By construction $q_\tau\geq 0$. Moreover, being $\tau\neq 1/2$, reasoning as in \cite[Lemma A.1]{FL23} one can prove that $q_\tau$ is closed. By the first representation theorem there exists a unique self-adjoint operator $T_\tau$ such that $\D(T_\tau)\subset \operatorname{dom}(q_\tau)=H^1_0\left((0,\rho);\C\right)$ and $\left\langle T_\tau u,u\right\rangle_{L^2((0,\rho);\C)}=q_\tau[u]$ for any $u\in\D(T_\tau)$. 
For these reasons, we get
\[
T_\tau u(r)=-\partial_r^2u(r)+\frac{\tau^2-\tau}{r^2}u(r).
\]
Being $T_\tau$ a positive operator, and denoting by $E_\tau$ its smallest eigenvalue, by the min-max theorem we get
\[
E_\tau = \inf_{ u\in \D(T_\tau)\setminus \{ 0 \} }
\frac{ \langle T_\tau u, u\rangle_{L^2((0,\rho);\C)}}{\norm{u}_{L^2((0,\rho);\C)}^2}=
\inf_{ u\in \D(T_\tau)\setminus \{ 0 \} }
\frac{q_\tau[u]}{\norm{u}_{L^2((0,\rho);\C)}^2}.
\]
Thanks to this and by \eqref{eq:dec-in-radial} we have that
\begin{equation}\label{eq:estimate-eingevalue}
\norm{D_\rho\psi}_{L^2(\Omega_\rho;\C^2)}^2
\geq
\inf_{k\in \mathbb{N}} E_{\pm \tau_k}^2.
\end{equation}
For $\tau\in\R\setminus\{1/2\}$ and for $E\geq 0$, the general solution of the equation $T_{\tau}u=E u$ is 
\[
u(r):=c_1 \sqrt{r} J_{|\tau-1/2|}\left(\sqrt{E} \, r\right)+
c_2 \sqrt{r}Y_{|\tau-1/2|}\left(\sqrt{E}\,  r\right),
\]
where $c_1,c_2\in \C$ and  $J$ and $Y$ are the Bessel functions of first kind. 
   To get $u\in H^1_0((0,\rho);\C^2)$, using the asymptotic expansion for $r\to 0$,
   \[
   J_{|\tau-1/2|}\left(\sqrt{E}\, r\right)\sim(\sqrt{E}\, r)^{|\tau-1/2|}\quad\text{and}\quad
   Y_{|\tau-1/2|}\left(\sqrt{E}\, r\right)\sim(\sqrt{E}\, r)^{-|\tau-1/2|},
   \]
we deduce that $c_2=0$. From the same expansion we deduce $E>0$, otherwise $u=0$. 
Lastly, ensuring $u(\rho)=0$ is necessary to achieve $H^1_0$-regularity.
This means that $J_{|\tau-1/2|}\left(\sqrt{E}\, \rho\right)=0$, and therefore 
\[
E_\tau= (j_{|\tau-1/2|,1})^2/\rho^2,
\]
being $j_{\nu,1}$ the first positive zero of the Bessel function of first kind $J_{\nu}$. 
Thanks to this and to \eqref{eq:estimate-eingevalue}, using  the inequality ${j_{k,1}}^2\geq {j_{0,1}}^2+k^2 $ proved in \cite{MCANN77} and the fact that ${j_{0,1}}^2\approx 5.78$, the proof is concluded.
\end{proof}


\end{document}